\documentclass[12pt,a4paper]{amsart}
\usepackage{a4wide}
\usepackage{amsmath, amssymb, amsfonts,enumerate}
\usepackage[all]{xy}
\usepackage{amscd}
\usepackage{comment}
\usepackage{mathtools}

\newtheorem{theorem}{Theorem}[section]
\newtheorem{lemma}[theorem]{Lemma}
\newtheorem{corollary}[theorem]{Corollary}
\newtheorem{proposition}[theorem]{Proposition}

 \theoremstyle{definition}
 \newtheorem{definition}[theorem]{Definition}
 \newtheorem{remark}[theorem]{Remark}

 \newtheorem{example}[theorem]{Example}
\newtheorem{examples}[theorem]{Examples}

\numberwithin{equation}{section}
\newcommand {\N}{\mathbb{N}} 
\newcommand {\Z}{\mathbb{Z}} 
\newcommand {\R}{\mathbb{R}} 

\newcommand{\T}{\mathbb{T}}

\newcommand{\CC}{\mathcal{C}}
\newcommand{\DD}{\mathcal{D}}

\newcommand{\FF}{\mathcal{F}}
\newcommand{\GG}{\mathcal{G}}

\newcommand{\PP}{\mathcal{P}}

\newcommand{\UU}{\mathcal{U}}
\newcommand{\VV}{\mathcal{V}}

   \DeclareMathOperator{\GL}{GL}
      \DeclareMathOperator{\SL}{SL}

\DeclareMathOperator{\mdim}{mdim}

\DeclareMathOperator{\LCA}{LCA}

\DeclareMathOperator{\End}{End}

\DeclareMathOperator{\M}{Mat}

\DeclareMathOperator{\Id}{Id}

\DeclareMathOperator{\supp}{supp}
\DeclareMathOperator{\Spec}{Spec}
\newcommand{\Proj}{\mathbb{P}}

\makeindex

\begin{document}
\title{The Garden of Eden theorem: old and new}
\author{Tullio Ceccherini-Silberstein}
\address{Universit\`a del Sannio, 82100 Benevento, Italy}
\email{tullio.cs@sbai.uniroma1.it}
\author{Michel Coornaert}
\address{Universit\'e de Strasbourg, CNRS, IRMA UMR 7501, F-67000 Strasbourg, France}
\email{michel.coornaert@math.unistra.fr}
\begin{abstract}
We review topics in the theory of cellular automata and dynamical systems that are related to the Moore-Myhill Garden of Eden theorem.
\end{abstract}
\date{\today}
\maketitle

\tableofcontents

\section{Introduction}

In the beginning, the Garden of Eden theorem, also known as the \emph{Moore-Myhill theorem}, is a result in the theory of cellular automata which states that a cellular automaton is surjective if and only if it satisfies a weak form of injectivity, called pre-injectivity.
The theorem was obtained by Moore\index{Moore, Edward F.} and Myhill\index{Myhill, John R.} in the early 1960s
 for cellular automata with finite alphabet over the groups $\Z^d$. 
The fact that surjectivity implies pre-injectivity for such cellular automata was first proved by Moore in~\cite{moore}, and  
 Myhill~\cite{myhill} established the converse implication shortly after.
 The proofs of Moore and Myhill appeared in two separate papers 
both published in 1963.
The biblical terminology used to designate the  Moore-Myhill theorem comes from the fact that configurations that are not in the image of a cellular automaton are called 
\emph{Garden of Eden configurations} because, when considering the sequence of consecutive iterates of the cellular automaton applied to the set of configurations,
they can only occur at time $0$.
Surjectivity of a cellular automaton is equivalent to absence of Garden of Eden configurations. 
In 1988~\cite[Question~1]{schupp-arrays}, Schupp\index{Schupp, Paul E.} asked whether the class of  groups for which the Garden of Eden theorem remains  valid is precisely the class of  virtually nilpotent groups.
By a celebrated result of Gromov~\cite{gromov-polynomial-growth},\index{Gromov, Misha L.}
a finitely generated group is virtually nilpotent 
if and only if it has polynomial growth.
In~1993,  Mach{\`{\i}} and Mignosi~\cite{machi-mignosi}
proved that the Garden of Eden theorem is still valid over any finitely generated group with subexponential growth.
As Grigorchuk~\cite{grigorchuk-1984},\index{Grigorchuk, Rostislav I.}
answering a longstanding open question raised by Milnor~\cite{milnor-growth},\index{Milnor, John W.}
gave   examples  of   groups whose growth lies strictly  between polynomial and exponential,  
it follows that the class of finitely generated  groups satisfying the Garden of Eden theorem is  larger than the class of finitely generated  virtually nilpotent groups.
Actually, it is  even  larger than the class of finitely generated groups with subexponential growth.
Indeed, Mach{\`{\i}},  Scarabotti, and the first author~\cite{ceccherini} proved in~1999 
that every  amenable group satisfies the Garden of Eden theorem
and it is a well--known fact that there are finitely generated amenable groups, such as the solvable Baumslag-Solitar group $BS(1,2) = \langle a,b:aba^{-1} = b^2\rangle$, that are amenable and have  exponential growth.
It was finally shown that the class of groups that satisfy the Garden of Eden theorem 
is precisely the class of amenable groups.
This is a consequence of recent results of Bartholdi~\cite{bartholdi}, Bartholdi\index{Bartholdi, Laurent} 
and Kielak~\cite{bartholdi:2016},
who showed that none of the implications of the Garden of Eden theorem
holds  when the group is nonamenable.
\par
In~\cite[Section~8]{gromov-esav},
Gromov made an important contribution to the subject 
by providing a deep analysis of the role played by entropy in the proof of the Garden of Eden theorem
and indicating new  directions for extending it  in many other interesting settings.
He mentioned  in particular
\cite[p~195]{gromov-esav}
the possibility of proving an analogue of the Garden of Eden theorem for a suitable class of hyperbolic dynamical systems.
Some results in that direction were subsequently obtained by the authors in~\cite{csc-goe-anosov},
\cite{csc-ijm-expansive}, \cite{csc-goe-principal} and, in collaboration with H.\ Li,\index{Li, Hanfeng} \cite{cscl-goe-homoclinically}.
In particular, a version of the Garden of Eden theorem was established    
for Anosov diffeomorphisms on tori in~\cite{csc-goe-anosov} and for principal expansive 
algebraic actions of countable abelian groups in~\cite{csc-goe-principal}.
\par
The present article is intended as a reasonably self-contained survey on the classical Garden of Eden theorem and some of its generalizations.    Almost all results presented here have already appeared in the literature elsewhere but  we sometimes give complete proofs when we feel they might be  helpful
to the reader.
The general theory of cellular automata over groups is developed in our monograph~\cite{book}. The present survey is a kind of complement to our book since for instance cellular automata between subshifts are not considered in \cite{book} while they are treated  here. 
 \par
The paper is organized as follows.
Configuration spaces and shifts are presented in Section~\ref{sec:configur-shifts}.
Cellular automata are introduced in Section~\ref{sec:ca}.
Section~\ref{sec:goe-Zd} contains the proof of the Garden of Eden theorem in the case $G = \Z^d$
following Moore and Myhill.
The proof of the Garden of Eden theorem in the  case of an arbitrary countable amenable group is given in Section~\ref{sec:goe-amenable}.
Examples of cellular automata that do not satisfy the Garden of Eden theorem for groups containing nonabelian free subgroups are described in Section~\ref{sec:goe-fails-non-amen}.
We also discuss the results of Bartholdi and Kielak mentioned above, which, together with the Garden of Eden theorem,
lead to characterizations of amenability in terms of cellular automata.
Extensions of  the Garden of Eden theorem to certain classes of   subshifts are
reviewed in Section~\ref{sec:goe-subshifts}.
In Section~\ref{sec:goe-dyn-sys},
we present versions of
the Garden of Eden theorem we obtained for certain classes of dynamical systems.
The final section briefly discusses some additional topics and provides references for further readings.    

\section{Configuration spaces and shifts}
\label{sec:configur-shifts}

\subsection{Notation}
We use the symbol $\Z$ to denote the set of integers $\{\dots,-2,-1,0,1,2,\dots\}$.
The symbol  $\N$  denotes the set of nonnegative integers $\{0,1,2,\dots\}$.
The cardinality of a finite set $X$ is written $|X|$.
\par
We  use multiplicative notation for groups  except for abelian groups such as 
\[
\Z^d = \underbrace{\Z \times \Z \times \cdots \times \Z}_{\footnotesize\mbox{$d$ times}}
\] 
for which we generally prefer additive notation.
\par 
Let $G$  be a group. We denote the identity element of $G$ by $1_G$.
If $A, B$ are subsets of $G$ and $g \in G$, we write
$A B \coloneqq \{ a b : a \in A,b \in B\}$,
$A^{-1} \coloneqq \{a^{-1} : a \in A\}$, $g A \coloneqq \{g\} A$ and $A g \coloneqq A \{g\}$. A subset $A \subset G$ is said to be \emph{symmetric}\index{symmetric subset} if it satisfies $A = A^{-1}$.

\subsection{Configurations spaces}
Let $\UU$ be a countable set, called the \emph{universe},\index{universe} and 
$A$  a finite set, called the \emph{alphabet}.\index{alphabet}
Depending on the context, the elements of $A$ are called  
\emph{letters}, or   \emph{symbols}, or \emph{states}, or \emph{colors}.  
As usual, we denote by $A^\UU$ the set consisting of all maps $x \colon \UU \to A$.
An element of $A^\UU$ is called a \emph{configuration}\index{configuration} of the universe $\UU$.
Thus, a configuration is a way of attaching a letter of the alphabet to each element of the universe.
\par
If $x \in A^\UU$ is a configuration and $\VV \subset \UU$, 
we shall write $x\vert_\VV$ for the \emph{restriction} of $x$ to $\VV$, i.e., the element 
$x|_\VV \in A^\VV$ defined by $x|_\VV(v) = x(v)$ for all $v \in \VV$. 
If $X \subset A^\UU$,
we shall write
\begin{equation}
\label{e:x-VV}
X_\VV \coloneqq \{ x\vert_\VV : x \in X \} \subset A^\VV.
\end{equation}
\par
Two configurations $x,y \in A^\UU$ are said to be \emph{almost equal}\index{almost equal configurations} if they coincide outside of a finite set, i.e., there is a finite subset
$\Omega \subset \UU$ such that $x|_{\UU \setminus \Omega} = y|_{\UU \setminus \Omega}$. 
Being almost equal clearly defines an equivalence relation on  $A^\UU$.
\par 
  We equip the configuration set $A^\UU$ with its  \emph{prodiscrete}\index{prodiscrete!--- topology}
  topology, that is,
the product topology obtained by taking the discrete topology on each factor
$A$ of $A^\UU = \prod_{u \in \UU} A$.
A neighborhood base of a configuration $x \in A^\UU$ is given by the sets
\begin{equation}
\label{e;cilinders}
V(x,\Omega) = V(x,\Omega,\UU,A) \coloneqq \{ y \in A^\UU : x\vert_\Omega = y\vert_\Omega \}, 
\end{equation}
where $\Omega$ runs over all finite subsets of $\UU$.
In this topology, two configurations  are ``close" if they coincide on a ``large" finite subset of  the universe. 
\par
Every finite discrete topological space is compact,    totally disconnected, and metrizable.
As a product of compact (resp.~totally disconnected) topological spaces is itself
compact (resp.~totally disconnected)
and a countable product of metrizable spaces is itself metrizable, 
it follows that $A^\UU$ is a compact totally disconnected metrizable space.
Note that $A^\UU$ is homeomorphic to the Cantor set as soon as $A$ contains more than one element and $\UU$ is infinite. 

\subsection{Group actions}
An \emph{action}\index{action} of a group $G$ on a set $X$ is a map
$\alpha \colon G \times X \to X$ satisfying
$\alpha(g_1,\alpha(g_2,x)) = \alpha(g_1 g_2,x)$ and $\alpha(1_G,x) = x$ for all $g_1,g_2 \in G$ and $x \in X$.
In the sequel, if $\alpha$ is an action of a group $G$ on a set $X$, we shall simply   write 
$g x$ instead of $\alpha(g,x)$, if there is no risk of confusion.
\par
Suppose that a group $G$ acts on a set $X$.
The \emph{orbit}\index{orbit} of a point $x \in X$ is the subset $G x \subset X$ defined by
$G x \coloneqq \{g x : g \in G\}$.
A point $x \in X$ is called \emph{periodic}\index{periodic point} if its orbit is finite.
A subset $Y \subset X$ is called \emph{invariant}\index{invariant subset}
if $G y \subset Y$ for all $y \in Y$.
One says that $Y \subset X$ is \emph{fixed}\index{fixed subset} by $G$ if $g y = y$ for all $g \in G$ and $y \in Y$.
\par 
Suppose now that a group $G$ acts on two sets $X$ and $Y$. A map 
$f \colon X \to Y$ is called \emph{equivariant}\index{equivariant map} 
if $f(g x) = g f(x)$ for all $g \in G$ and $x \in X$. 
\par
Let $X$ be a topological space. An action of a group $G$ on $X$ is called \emph{continuous}\index{continuous action}\index{action!continuous ---} if the map $x \mapsto g x$ is continuous on $X$ for each $g \in G$.
Note that if $G$ acts continuously on $X$ then, for each $g \in G$,  the map $x \mapsto g x$ is a homeomorphism of $X$ with inverse $x \mapsto g^{-1} x$. 

\subsection{Shifts}
From now on, our universe will be a countable group.
So let $G$ be a countable group and $A$ a finite set.
Given an element $g \in G$  and a configuration $x \in A^G$, we define
the configuration $gx \in A^G$ by
\begin{equation*}
\label{e:decalagge}
  gx \coloneqq x \circ L_{g^{-1}},
\end{equation*}
where $L_g \colon G \to G$ is the left-multiplication by $g$.
Thus
$$
gx (h) =   x (g^{-1}h) \quad \text{for all } h \in G.
$$
Observe that, for all $g_1, g_2 \in G$ and $x \in A^G$,
\begin{equation*}
g_1(g_2x) =   x \circ L_{g_2^{-1}} \circ L_{g_1^{-1}} = x \circ L_{g_2^{-1}g_1^{-1}} = x \circ L_{(g_1g_2)^{-1}} 
= (g_1g_2)x,
\end{equation*}
and 
$$
1_G x = x \circ L_{1_G} = x \circ \Id_G = x.
$$
Therefore the map
\begin{align*}
G \times A^G &\to A^G \\
(g,x) &\mapsto gx 
\end{align*}
defines an action of $G$ on $A^G$. 
This action is called the $G$-\emph{shift}, or simply the \emph{shift},\index{shift}  on 
 $A^G$.
\par
Observe that if two configurations $x,y \in A^G$ coincide on a subset $\Omega \subset G$, 
then, for every $g \in G$,
the configurations $g x$ and $g y$ coincide on $g \Omega$.
As the sets $V(x,\Omega)$ defined by~\eqref{e;cilinders} 
    form a base of neighborhoods of $x \in A^G$ when $\Omega$ runs over all finite subsets of $G$, 
    we deduce that
the map $x \mapsto g x$ is continuous  on $A^G$ for each 
$g \in G$. Thus, the shift action of $G$ on $A^G$ is continuous. 

\subsection{Patterns}
A \emph{pattern}\index{pattern} is a map $p \colon \Omega \to A$, where $\Omega$ is a finite subset of $G$.
If $p \colon \Omega \to A$ is a pattern,
we say that $\Omega$ is the \emph{support}\index{support!--- of a pattern}\index{pattern!support of a ---} of $p$ and write $\Omega = \supp(p)$.
\par
Let $\PP(G,A)$ denote the set of all patterns.
There is a natural  action of the group $G$ on $\PP(G,A)$ defined as follows.
Given $g \in G$ and a pattern $p \in \PP(G,A)$ with support $\Omega$,
we define the pattern $g p \in \PP(G,A)$ as being the pattern with support $g \Omega$ such that
$g p(h) = p(g^{-1} h)$ for all $h \in g \Omega$.
It is easy to check that this defines  an  action of $G$ on $\PP(G,A)$, i.e.,
$g_1(g_2 p) = (g_1 g_2) p$ and $1_G p = p$ for all $g_1, g_2 \in G$ and $p \in \PP(G,A)$.
Observe that
$\supp(g p) = g \supp(p)$ for all $g \in G$ and $p \in \PP(G,A)$.
Note also that if $p$ is the restriction of a configuration $x \in A^G$ to a finite subset 
$\Omega \subset G$,
then $g p$ is the restriction of the configuration $g x$ to $g \Omega$.

\begin{example}
\label{ex:patterns-are-words}
Take $G = \Z$  and let $A$ be a finite set.
Denote by $A^\star$ the set of words on the alphabet $A$.\index{word}
We recall that $A^\star$ is the free monoid based on $A$ and that any element $w \in A^\star$ can be uniquely written in the form $w = a_1 a_2 \cdots a_n$, where $a_i \in A$ for $1 \leq i \leq n$ and   $n \in \N$ is the \emph{length}\index{length!--- of a word}\index{word!length of a ---} of the word $w$. The monoid operation on $A^\star$ is the concatenation of words and the identity element  is the empty word, that is, the unique word with length $0$.
Now let us fix some finite interval $\Omega \subset \Z$  of cardinality $n$, say
$\Omega = \{m , m + 1, \dots, m + n - 1\}$ with $m \in \Z$ and $n \in \N$.
Then one can associate with each pattern $p \colon \Omega \to A$ 
the word 
\[
w = p(m) p(m + 1) \cdots p(m + n - 1) \in A^\star.
\]
This yields a one-to-one correspondence between the patterns supported by $\Omega$ and the words of length $n$ on the alphabet $A$.
This is frequently used to identify each pattern supported by $\Omega$ with the corresponding word.
\end{example}

\subsection{Subshifts}
A \emph{subshift}\index{subshift} is a subset 
$X \subset A^G$ that is invariant under the $G$-shift and closed for the prodiscrete topology on $A^G$.

\begin{example}
\label{ex:golden-sub}
Take $G = \Z$ and $A = \{0,1\}$.
Then the  subset $X \subset A^G$, consisting of all $x \colon \Z \to \{0,1\}$ such that
$(x(n),x(n + 1)) \not= (1,1)$ for all $n \in \Z$, is a subshift.
This subshift is called the \emph{golden mean subshift}.\index{golden mean subshift}\index{subshift!golden mean ---}
\end{example}

\begin{example}
\label{ex:hard-ball}
Take $G =\Z^d$ and $A = \{0,1\}$.
Then the  subset $X \subset A^G$, consisting of all $x \in A^G$ such that
$(x(g),x(g + e_i)) \not= (1,1)$ for all $g \in G$, where $(e_i)_{1\leq i \leq d}$ is the canonical basis of $\Z^d$,
is a subshift.
This subshift is called the \emph{hard-ball model}.\index{hard-ball model}\index{subshift!hard-ball model ---}
For $d = 1$, the hard-ball model is  the golden mean subshift of the previous example. 
\end{example}

\begin{example}
\label{ex:even-subshift}
Take $G = \Z$ and $A = \{0,1\}$.
Then the subset $X \subset A^G$, consisting of all bi-infinite sequences $x \colon \Z \to \{0,1\}$ such that there is always an even number of $0$s between two $1$s, is a subshift.
This subshift is called the \emph{even subshift}.\index{even subshift}\index{subshift!even ---}
\end{example}

\begin{example}
\label{ex:ledrappier-sub}
Take $G = \Z^2$ and $A = \{0,1\} = \Z/2\Z$ (the integers modulo $2$).
Then the  subset $X \subset A^G$, consisting of all $x \colon \Z^2 \to \{0,1\}$ such that
\[
x(m,n) + x(m + 1,n) + x(m,n + 1) = 0
\]
for all $(m,n) \in \Z^2$, is a subshift.
This subshift is called the \emph{Ledrappier subshift}.\index{Ledrappier subshift}\index{subshift!Ledrappier ---}
\end{example}

\begin{remark}
Every intersection  of subshifts and every finite union of subshifts $X \subset A^G$ is itself a subshift. Therefore the subshifts $X \subset A^G$ are the closed subsets of a topology on $A^G$. This topology is coarser (it has less open sets) than the prodiscrete topology on $A^G$.
It is not Hausdorff as soon as $G$ is not trivial and $A$ has more than one element. 
\end{remark}

Given a (possibly infinite) subset of patterns $P \subset \PP(G,A)$,
it is easy to see that the subset $X(P) \subset A^G$ defined by
$$
X(P) \coloneqq \{ x \in A^G \text{ such that  }  (g x)|_{\supp(p)} \not= p \text{ for all  } g \in G \text{ and } p \in P \}
$$
is a subshift.
\par
Conversely, let  $X \subset A^G$ be a subshift.
One says that a pattern $p \in \PP(G,A)$ 
\emph{appears}\index{pattern!--- appearing in a subshift}\index{subshift!pattern appearing in a ---} in $X$ if
$p \in X_{\supp(p)}$, i.e., if   there is a configuration $x \in X$
such that $x|_{\supp(p)} = p$.
Then one easily checks that $X = X(P)$ for
$$
P \coloneqq  \{ p \in \PP(G,A) \text{ such that  }   p \text{ does not appear in   }  X \}.
$$
\par
One says that a subshift $X \subset A^G$ is \emph{of finite type}\index{finite type!subshift of ---}\index{subshift!--- of finite type} if there exists a finite subset 
$P \subset \PP(G,A)$ such that $X = X(P)$.
The  hard-ball models (and hence in particular the golden mean subshift) and the Ledrappier subshift are examples of subshifts of finite type. On the other hand, the even subshift is not of finite type. 

\section{Cellular automata}
\label{sec:ca}

\subsection{Definition}

Let  $G$ be a countable group and let $A, B$ be  finite sets.
Suppose that $X \subset A^G$ and $Y \subset B^G$ are two subshifts.

\begin{definition}
\label{def:ca}
One says that a map $\tau \colon X \to Y$ is a
\emph{cellular automaton}\index{cellular automaton}
if there exist a finite subset $S \subset G$
and a map $\mu \colon A^S \to B$ such that
\begin{equation}
\label{def:automate}
 \tau(x)(g) = \mu((g^{-1}x)\vert_S)
\end{equation}
for all $x \in X$ and $g \in G$, where we recall that $(g^{-1}x)\vert_S$ denotes the restriction of the configuration $g^{-1}x \in X$ to $S$.
 Such a set $S$ is called a \emph{memory set}\index{memory set}\index{cellular automaton!memory set of a ---}
 and $\mu$ is called a \emph{local defining map}\index{local defining map}\index{cellular automaton!local defining map of a ---}
 for  $\tau$.
\end{definition}

It immediately follows from this definition that a map $\tau \colon X \to Y$  
is a cellular automaton
if and only if $\tau$ extends to a cellular automaton $\widetilde{\tau} \colon A^G \to B^G$.
Observe also that if $S$ is a memory set for a cellular automaton $\tau \colon X \to Y$ and $g \in G$,
then Formula~\eqref{def:automate} implies that the value taken by the configuration $\tau(x)$ at $g$ only depends on the restriction of $x$ to $gS$.
 Finally note that 
if $S$ is a memory set for a cellular automaton $\tau$, then any finite subset of $G$ containing $S$ is also a memory set for $\tau$.
Consequently, the memory set of a cellular automaton is not unique in general.
However, every cellular automaton admits a unique memory set with minimal cardinality
(this follows from the fact that if $S_1$ and $S_2$ are memory sets then so is $S_1 \cap S_2$).

\begin{example}
\label{ex:ca-sum-mod-2}
Take $G = \Z$ and $A = \{0,1\} = \Z/2\Z$.
Then the map $\tau \colon A^G \to A^G$, defined by
$$
\tau(x)(n) \coloneqq x(n + 1) + x(n)
$$
for all $x \in A^G$ and $n \in \Z$, is a cellular automaton admitting
$S \coloneqq \{0,1\} \subset \Z$ as a memory set and $\mu \colon A^S \to A$ given by
$$
\mu(p) \coloneqq  p(0) + p(1)
$$
for all $p \in A^S$, as a local defining map.
Using the representation of patterns with support $S$ by words of length $2$ on the alphabet $A$ (cf.~Example~\ref{ex:patterns-are-words}), the map $\mu$ is given by
\[
\mu(00) = \mu(11) = 0 \text{ and }   \mu(01) =  \mu(10) = 1. 
\]
\end{example}

\begin{example}[Majority vote]\index{majority vote}\index{cellular automaton!majority vote ---}
\label{ex:majority-ca}
Take $G = \Z$ and $A = \{0,1\}$.
Then the map $\tau \colon A^G \to A^G$, defined by
\[
\tau(x)(n) \coloneqq
\begin{cases} 
0 & \text{ if } x(n-1) + x(n) + x(n + 1) \leq 1 \\
1 & \text{ otherwise}
\end{cases}
\]
for all $x \in A^G$ and $n \in \Z$, is a cellular automaton admitting
$S \coloneqq \{-1,0,1\} \subset \Z$ as a memory set and $\mu \colon A^S \to A$ given by
\[
\mu(000) = \mu(001) = \mu(010) = \mu(100) = 0
\]
 and
 \[ 
\mu(011) = \mu(101) = \mu(110) = \mu(111) = 1
\]
as local defining map. The cellular automaton $\tau$ is called the \emph{majority vote} cellular automaton.
\end{example}

\begin{remark}
A cellular automaton $\tau \colon A^G \to A^G$, where $G = \Z$, $A = \{0,1\}$, admitting $S = \{-1,0,1\}$ as a memory set  is called an \emph{elementary cellular automaton}.\index{elementary!--- cellular automaton}\index{cellular automaton!elementary ---}
Each one of these  cellular automata is uniquely determined by its local defining map
$\mu \colon A^S \to A$, so that there are exactly $2^8 = 256$
elementary cellular automata.
They are numbered from $0$ to $255$ according to a notation that was introduced by Wolfram (cf.~\cite{wolfram-new-kind}).\index{Wolfram, Stephen}
To obtain the number $n$ of an elementary cellular automaton $\tau$, one proceeds as follows.
One first lists all eight patterns $p \in A^S$ in increasing order from $000$ to $111$.
The number $n$ is the integer whose expansion in base $2$ is $a_8 a_7 \dots a_1$, where $a_k$ is the value taken by the local defining map of $\tau$ at the $k$-th pattern in the list.
One also says that $\tau$ is  \emph{Rule} $n$.\index{Rule of an elementary cellular automaton}\index{elementary!Rule of an --- cellular automaton}
For instance, the elementary cellular automaton described in Example~\ref{ex:ca-sum-mod-2}
is Rule~102 while the one described in Example~\ref{ex:majority-ca} is Rule 232.\index{majority vote}\index{cellular automaton!majority vote ---}
\end{remark}

\begin{example}
\label{ex:identity-ca}
Let $G$ be a countable group, $A$ a finite set, and $X \subset A^G$ a subshift.
Then the identity map $\Id_X \colon X \to X$ is a cellular automaton with memory set
$S = \{1_G\}$ and local defining map $\mu = \Id_A  \colon A^S = A^{\{1_G\}} = A \to A$.
\end{example}

\begin{example}
Let $G$ be a countable group, $A$ a finite set, and $X \subset A^G$ a subshift.
Let $s \in G$ and denote by $R_s$ the right-multiplication by $s$, that is, the map
$R_s \colon G \to G$ defined by $R_s(h) \coloneqq h s$ for all $s \in G$.
Then the subset $Y \subset A^G$ defined by
$$
Y \coloneqq \{x \circ R_s \text{ such that } x \in X \}
$$
is a subshift.
Moreover, the map $\tau \colon  X \to Y$, defined by
$\tau(x) \coloneqq x \circ R_s$ for all $x \in X$,
is a cellular automaton with memory set $S = \{s\}$ and local defining map
$\mu = \Id_A \colon A^S = A \to A$.
Observe that if $s$ is in the center  of $G$, then $X =  Y$ and $\tau \colon X \to X$ is the shift map 
$x \mapsto s^{-1} x$. 
\end{example}

\begin{example}
\label{ex:ca-golden-even}
Take $G = \Z$ and  $A = \{0,1\}$.
Let $X \subset A^G$ and $Y \subset A^G$ denote respectively the golden mean subshift\index{golden mean subshift}\index{subshift!golden mean ---} and the even subshift.\index{even subshift}\index{subshift!even ---}
For $x \in X$, define $\tau(x) \in A^G$ by
$$
\tau(x)(n) \coloneqq
\begin{cases}
0 & \text{ if } (x(n),x(n + 1)) = (0,1) \text{ or } (1,0) \\
1& \text{ if } (x(n),x(n+1)) = (0,0)
\end{cases}
$$
for all $n \in \Z$. It is easy to see that $\tau(x) \in Y$ for all $x \in X$.
The map $\tau \colon X \to Y$ is a cellular automaton admitting $S \coloneqq \{0,1\} \subset \Z$ as a memory set and the map
$\mu \colon A^S  \to A$, defined by
\[
\mu(00) = \mu(11) = 1 \text{ and } \mu(01) = \mu(10) = 0.
\]
Note that the map $\mu' \colon A^S \to A$, defined by
\[
\mu'(00) = 1 \text{ and } \mu'(01) = \mu'(10) = \mu'(11) =  0
\]
is also a local defining map for $\tau$.
Thus, $\tau$ is the restriction to $X$ of both Rule $153$ and Rule $17$.
\end{example}

\subsection{The Curtis-Hedlund-Lyndon theorem}
The definition of a cellular automaton given in the previous subsection  is a local one. For $\tau$ to be a cellular automaton,
it requires the existence of a rule, commuting with the shift, that 
allows one to evaluate the value taken by 
$\tau(x)$ at  $g \in G$ by applying the rule to the restriction of $x$ to a certain finite set, namely  
the left-translate by $g$ of a memory set of the automaton.
The following result, known as the \emph{Curtis-Lyndon-Hedlund theorem} (see~\cite{hedlund}),  yields a global characterization of cellular automata involving only the shift actions and the prodiscrete topology on the configuration spaces.

\begin{theorem}
\label{t:curtis-hedlund-lyndon}\index{Curtis-Hedlund-Lyndon theorem}\index{theorem!Curtis-Hedlund-Lyndon ---}
Let  $G$ be a countable group and let $A,B$ be finite sets.
Let $\tau \colon X \to Y$ be a map from a subshift $X \subset A^G$ into 
a subshift $Y \subset B^G$. 
  Then the following conditions are equivalent:
\begin{enumerate}[\rm (a)]
\item
$\tau$ is a cellular automaton;
\item
$\tau$ is equivariant (with respect to the shift actions of $G$) and continuous (with respect to the prodiscrete topologies).   
\end{enumerate}
\end{theorem}

\begin{proof}
Suppose first that   $\tau \colon X \to Y$ is a cellular automaton.
Let $S \subset G$ be a memory set  and $\mu \colon A^S \to B$ a local defining map for $\tau$.
For all $g,h \in G$
and $x \in X$, we have that
\begin{align*}
\tau(gx)(h) &= \mu((h^{-1}gx)\vert_S) && \text{(by Formula~\eqref{def:automate})} \\
&= \mu(((g^{-1}h)^{-1}x)\vert_S) \\
&= \tau(x)(g^{-1}h) \\
&= g\tau(x)(h).
\end{align*}
Thus $\tau(gx) = g\tau(x)$ for all $g \in G$ and $x \in X$.
This shows that $\tau$ is equivariant.
\par
Now let $\Omega$ be a finite subset of $G$.
Recall that  Formula~\eqref{def:automate}
implies that if two configurations $x,  y \in X$ coincide on $gS$ for some $g \in G$,
then $\tau(x)(g) = \tau(y)(g)$.
Therefore, if the configurations $x$ and $y$ coincide on the finite set
$ \Omega  S$,
then $\tau(x)$ and $\tau(y)$ coincide on $\Omega$.
It follows that
$$
\tau(X \cap V(x,\Omega S,G,A)) \subset V(\tau(x),\Omega,G,B). 
$$
This implies that   $\tau$ is continuous.
Thus  (a) implies (b).
\par
Conversely, suppose now that the map  $\tau \colon X \to Y$ is equivariant and continuous.
 Let us show that $\tau$ is a cellular automaton.
 As the map $\varphi \colon X \to B$ defined by
$\varphi(x) \coloneqq \tau(x)(1_G)$ is continuous, we can find, for each $x \in X$, a finite subset 
$\Omega_x \subset G$
such that if $y \in X \cap V(x,\Omega_x,G,A)$, then $\tau(y)(1_G) =
\tau(x)(1_G)$. The sets
$X \cap V(x,\Omega_x,G,A)$ form an open cover of $X$. As $X$ is compact, there is a finite subset
$F \subset X$ such that the sets $V(x,\Omega_x,G,A)$, $x \in F$, cover $X$.
Let us set $S = \cup_{x \in F} \Omega_x$ and suppose that two configurations $y,z \in X$ coincide on $S$.
Let $x_0 \in F$ be such that $y \in V(x_0, \Omega_{x_0},G,A)$, that is, $y \vert_{\Omega_{x_0}} = x_0\vert_{\Omega_{x_0}}$.
As $\Omega_{x_0} \subset S$, we have that
$y \vert_{\Omega_{x_0}} = z \vert_{\Omega_{x_0}}$ and therefore  $\tau(y)(1_G) = \tau(x_0)(1_G) = \tau(z)(1_G)$.
We deduce that  there exists  a map $\mu \colon A^S \to B$ such that $\tau(x)(1_G) = \mu(x\vert_S)$ for all $x \in X$.
Now, for all $x \in X$ and  $g \in G$, we have that
\begin{align*}
\tau(x)(g) &= (g^{-1} \tau(x))(1_G) && \text{(by definition of the shift action on $B^G$)} \\
&= \tau(g^{-1}x)(1_G) &&\text{(since $\tau$ is equivariant)} \\
&= \mu((g^{-1} x)|_S).
\end{align*}
This shows that   $\tau$ is a cellular automaton
with memory set $S$ and local defining map $\mu$.
Thus (b) implies (a).
\end{proof}

\subsection{Operations on  cellular automata}

\begin{proposition}
\label{p:composition-ca}
Let $G$ be a countable group and let $A, B,C$ be finite sets.
Suppose  that $X \subset A^G$, $Y \subset B^G$, $Z \subset C^G$ are subshifts and that
$\tau \colon X \to Y$, $\sigma \colon Y \to Z$ are cellular automata.
Then the composite map $\sigma \circ \tau \colon X \to Z$ is a cellular automaton. 
\end{proposition}

\begin{proof}
This is an immediate consequence of the characterization of cellular automata 
given by the Curtis-Hedlund-Lyndon theorem
(cf.~Theorem~\ref{t:curtis-hedlund-lyndon}) since the composite of two equivariant (resp.~continuous) maps is itself equivariant (resp.~continuous).
\end{proof}

\begin{remark}
\label{rem:category-subshifts}
If we fix the countable group $G$, we deduce from Proposition~\ref{p:composition-ca} and Example~\ref{ex:identity-ca}
that the subshifts 
$X \subset A^G$, with $A$ finite,
are the objects of a concrete category $\CC_G$ in which  the set of morphisms from
 $X \in \CC_G$ to  $Y \in \CC_G$ consist of all cellular automata 
$\tau \colon X \to Y$ (cf.~\cite[Section~3.2]{csc-cat}). In this category, the endomorphisms of $X \in \CC_G$ consist of all cellular automata 
$\tau \colon X \to X$ and they form a monoid for the composition of maps. 
\end{remark}

\begin{proposition}
Let $G$ be a countable group and let $A, B$ be finite sets.
Suppose  that $X \subset A^G$, $Y \subset B^G$ are subshifts and that
$\tau \colon X \to Y$ is a bijective cellular automaton.
Then the inverse map 
$\tau^{-1} \colon Y \to X$ is a cellular automaton. 
\end{proposition}

\begin{proof}
This is again an immediate consequence of the  Curtis-Hedlund-Lyndon theorem
since the inverse  of a bijective  equivariant map is an  equivariant map
and the inverse of a bijective continuous map between compact Hausdorff spaces is  continuous. 
\end{proof}

\subsection{Surjectivity of cellular automata, GOE configurations, and   GOE patterns}
In what follows, we keep the notation introduced for defining cellular automata. 
Let $\tau \colon X \to Y$ be a cellular automaton.
A configuration $y \in Y$ is called a \emph{Garden of Eden}\index{Garden of Eden!--- configuration}\index{configuration!Garden of Eden ---} configuration for $\tau$,
briefly a \emph{GOE configuration}, if it does not belong to the image of $\tau$,
i.e., there is no $x \in X$ such that $y = \tau(x)$.
One says that a pattern $p \in \PP(G,B)$ is a \emph{Garden of Eden pattern}\index{Garden of Eden!--- pattern}\index{pattern!Garden of Eden ---} 
for $\tau$, briefly a \emph{GOE pattern},
 if the pattern $p$ appears in the subshift $Y$ but not  in the  subshift  $\tau(X)$, i.e., 
there exists $y \in Y$ such that $p = y|_{\supp(p)}$ but there is no $x \in X$ such that
 $p = \tau(x)|_{\supp(p)}$.
 Note that the set of GOE configurations (resp.~of GOE patterns) is an invariant subset of
 $Y$ (resp.~of $\PP(G,B)$). 
 Observe also that if $p \in \PP(G,B)$ is a GOE pattern for the cellular automaton $\tau \colon X \to Y$, 
then every configuration $y \in Y$ such that $y|_{\supp(p)} = p$ is a GOE configuration for $\tau$.

\begin{example}
It is easy to check that the pattern with support $\Omega \coloneqq \{0,1,2,3,4\}$ associated with the word $01001$ is a GOE pattern for the majority vote cellular automaton described in Example~\ref{ex:majority-ca}.\index{majority vote}\index{cellular automaton!majority vote ---}
\end{example}

\begin{proposition}
 \label{p:surj-equiv-no-goe}
  Let $\tau  \colon X \to Y$ be a cellular automaton.
 Then the  following conditions are  equivalent:
 \begin{enumerate}[\rm (a)]
 \item
 $\tau$ is  surjective;
 \item
 $\tau$ admits no GOE configurations;
 \item
 $\tau$ admits no GOE patterns.
\end{enumerate}
 \end{proposition}

\begin{proof}
The equivalence of (a) and (b) as well as the implication (b) $\implies$ (c) are trivial.
The implication (c) $\implies$ (b) easily follows from the compactness of $\tau(X)$.
\end{proof}

\subsection{Pre-injectivity of cellular automata and mutually erasable patterns}
Let $G$ be a countable group and let $A$ and $B$ be finite sets.
Recall that two configurations $x,y \in A^G$ are called almost
equal if they coincide outside of a finite subset of $G$.

\begin{definition}
\label{def:pre-inj-subshifts}
Let $X \subset A^G$ and $Y \subset B^G$ be subshifts.
One says that a cellular automaton $\tau \colon X \to Y$ is \emph{pre-injective}\index{pre-injective!--- cellular automaton}\index{cellular automaton!pre-injective ---} if
there are no distinct configurations $x_1, x_2 \in X$ that are almost equal and satisfy $\tau(x_1) = \tau(x_2)$.
\end{definition}

A pair of configurations  $(x_1,x_2)  \in X \times X$
is called a \emph{diamond}\index{diamond}  
if $x_1$ and $x_2$ are distinct,  almost equal,  and have the same image under $\tau$
(cf.~\cite[Definition~8.1.15]{lind-marcus}).
Thus, pre-injectivity is equivalent to absence of diamonds.
If $(x_1,x_2)$ is a diamond, the nonempty finite subset
\[
\{g \in G: x_1(g) \not= x_2(g) \} \subset G
\]
is called the \emph{support}\index{support!--- of a diamond}\index{diamond!support of a ---} of the diamond $(x_1,x_2)$.
\par
Every injective cellular automaton is clearly pre-injective.
The converse is false, as shown by the following examples. 

\begin{example}
Take $G = \Z$, $A = \{0,1\} = \Z/2\Z$, and consider the cellular automaton $\tau \colon A^G \to A^G$ described in Example~\ref{ex:ca-sum-mod-2} (Rule~102 in Wolfram's notation).
Then $\tau$ is pre-injective.
Indeed, it is clear that if two configurations $x, x' \in A^G$ coincide on $\Z \cap (-\infty,n_0]$ for some $n_0 \in \Z$ and satisfy $\tau(x) = \tau(x')$ then $x = x'$.
However, $\tau$ is not injective since the two constant configurations have the same image.
\end{example}

\begin{example}
\label{ex:ca-golden-even-pre-inj}
Consider the cellular automaton $\tau \colon X \to Y$ from the golden mean subshift to the even subshift described in Example~\ref{ex:ca-golden-even}.\index{golden mean subshift}\index{subshift!golden mean ---}\index{even subshift}\index{subshift!even ---}
It is easy to see that $\tau$ is pre-injective by an argument similar to the one used in the previous example.
However, $\tau$ is not injective since the two sequences in $X$ with exact period $2$ have the same image under $\tau$, namely the constant sequence with only $0$s. 
\end{example}

Let $\Omega \subset G$ be a finite set and $p_1, p_2 \in X_\Omega$ 
two   patterns with support $\Omega$ appearing in $X$. 
One says that the patterns $p_1$ and $p_2$ are \emph{mutually erasable}\index{mutually erasable patterns} with respect to $\tau$,
briefly \emph{ME},
 provided the following hold:
\begin{enumerate}[\rm (MEP-1)] 
\item 
the set
\[
X_{p_1,p_2} \coloneqq  \{(x_1, x_2) \in X \times X:
x_1 \vert_\Omega = p_1, x_2 \vert_\Omega = p_2 \mbox{ and } x_1 \vert_{G \setminus \Omega} = x_2 \vert_{G \setminus \Omega}\}
\]
is nonempty;
\item 
for all $(x_1, x_2) \in  X_{p_1,p_2}$ one has $\tau(x_1) = \tau(x_2)$.
\end{enumerate}
Note that ``being ME" is an equivalence relation on $X_\Omega$. This equivalence relation is not trivial in general.
\begin{example}
The patterns with support $\Omega \coloneqq \{0,1,2\}$ associated with the words $00000$ and $00100$ are ME patterns for the majority vote cellular automaton described in Example~\ref{ex:majority-ca}.\index{majority vote}\index{cellular automaton!majority vote ---}
\end{example}

Observe that  if the patterns $p_1$ and $ p_2$
are ME, then so are $gp_1$ and $gp_2$ for all $g \in G$.

\begin{proposition}
\label{p:MEP-preinjectivity}
Let $\tau \colon X \to Y$ be a cellular automaton.
Then the  following conditions are equivalent:
\begin{enumerate}[\rm (a)]
\item
$\tau$ is pre-injective;
\item
$\tau$ admits no distinct ME patterns.
\end{enumerate}
\end{proposition}

\begin{proof} 
Suppose first that $\tau$ admits two distinct ME patterns $p_1$ and $p_2$. 
Let $\Omega$ denote their  common support.
Take  $(x_1, x_2) \in  X_{p_1,p_2}$.
Then the configurations $x_1$ and $x_2$  are almost equal since they coincide outside of $\Omega$. 
Moreover, they satisfy  $x_1 \not= x_2$ and $\tau(x_1) = \tau(x_2)$.
Therefore $(x_1,x_2)$ is a diamond for $\tau$.
It follows that $\tau$ is not pre-injective.
\par
Suppose now that $\tau$ is not pre-injective and let us show that $\tau$ admits two distinct ME patterns.
By definition, $\tau$ admits a diamond $(x_1,x_2) \in X \times X$. 
Let $\Delta$ denote the support of this diamond  
and  let $S \subset G$ be a memory set for $\tau$ with $1_G \in S$. 
Consider the set  $\Omega \coloneqq \Delta S^{-1} S$.
 We claim that the patterns $p_1 \coloneqq x_1\vert_\Omega$ and $p_2 \coloneqq x_2\vert_\Omega$ are distinct ME patterns.
First observe that $p_1 \neq p_2$ since $\varnothing \not= \Delta \subset \Omega$. 
Moreover, $X_{p_1,p_2}$ is nonempty since, by construction, $(x_1, x_2) \in X_{p_1,p_2}$.
To complete  the proof, we only need to show the following: if $(y_1, y_2) \in X_{p_1,p_2}$ 
then $\tau(y_1) = \tau(y_2)$.
Let $g \in G$. Suppose first that $g \in G \setminus \Delta S^{-1}$. Then $gS \cap \Delta = \varnothing$. Since $y_1$ and $y_2$ coincide on $G \setminus \Delta$,
we deduce that
\begin{equation}
\label{e:egales-en-dehors}
\tau(y_1)(g) = \tau(y_2)(g) \text{  for all } g \in G \setminus \Delta S^{-1}.
\end{equation}
Suppose now that $g \in \Delta S^{-1}$. Then $gS \subset \Delta S^{-1}S = \Omega$.
As $y_1\vert_\Omega = p_1 = x_1\vert_\Omega$ (resp.\ $y_2\vert_\Omega = p_2 = x_2\vert_\Omega$), by construction, we deduce that
\begin{equation}
\label{e:egales-dedans}
\tau(y_1)(g) =  \tau(x_1)(g) = \tau(x_2)(g) = \tau(y_2)(g) \text{  for all } g \in \Delta S^{-1}. 
\end{equation}
From \eqref{e:egales-en-dehors} and \eqref{e:egales-dedans} we deduce that $\tau(y_1) = \tau(y_2)$.
\end{proof}

\section{The Garden of Eden theorem for $\Z^d$}
\label{sec:goe-Zd}

In this section, we present a proof of the Garden of Eden theorem of Moore and Myhill for cellular automata over the group $\Z^d$. This is a particular case of the Garden of Eden theorem for amenable groups that will be established in the next section. 

\begin{theorem}
\label{t:GOE-zd}\index{Garden of Eden theorem!--- (for $\Z^d$)}\index{theorem!Garden of Eden --- (for $\Z^d$)}
Let $A$ be a finite set, $d \geq 1$ an integer, and $\tau \colon A^{\Z^d} \to A^{\Z^d}$ a cellular automaton.
Then the following conditions are equivalent:
\begin{enumerate}[{\rm (a)}]
\item $\tau$ is surjective;
\item $\tau$ is pre-injective.
\end{enumerate}
\end{theorem}

The implication (a) $\implies$ (b) is due to Moore~\cite{moore}\index{Moore, Edward F.} and the converse to Myhill~\cite{myhill}.\index{Myhill, John R.}
We shall prove the contraposite of each implication.
Before undertaking the proof of Theorem \ref{t:GOE-zd}, let us first introduce some notation and establish some preliminary results.
\par
Let $S\subset \Z^d$ be a memory set for $\tau$.
Since any finite subset of $\Z^d$ containing a memory set for $\tau$ is itself a memory set for $\tau$, it is not restrictive to suppose that $S =\{0,\pm 1,\ldots, \pm r\}^d$ for some integer $r \geq 1$.
\par
Let us set, for each integer $m \geq 2r$,
\[
\Omega_m \coloneqq \{0,1,\ldots, m-1\}^d
\]
\[
\Omega_m^+ \coloneqq \{-r,-r+1,\ldots, m+r-1\}^d
\]
and 
\[
\Omega_m^- \coloneqq \{r,r+1,\ldots, m-r-1\}^d,
\]
so that
\begin{equation}
\label{e:size-cubes}
|\Omega_m| = m^d, \ \ \ |\Omega_m^+| = (m+2r)^d, \ \mbox{ and } \ |\Omega_m^-| = (m-2r)^d.
\end{equation}
Observe that if two configurations $x_1,x_2\in A^{\Z^d}$ coincide on $\Omega_m$ 
(resp.\ $\Omega_m^+$, resp.\ $\Z^d \setminus \Omega_m^-$) then
$\tau(x_1)$ and $\tau(x_2)$ coincide on $\Omega_m^-$ (resp.\ $\Omega_m$, resp.\ $\Z^d \setminus \Omega_m$).
\par
Also, let us set, for all  integers $k,n \geq 1$
\[
T_n^k \coloneqq \{t=(t_1k, t_2k, \ldots, t_dk)\in \Z^d: 0 \leq t_j \leq n-1\},
\]
and observe that the $n^d$ cubes  $t + \Omega_k$, for $t$ running over $T_n^k$,
form a partition of the cube $\Omega_{n k}$.
\par
Finally, we introduce the following additional notation. 
Given a finite subset $\Omega \subset G$  and  $p,q \in A^\Omega$ 
we write $p \sim q$ if and only if the patterns $p$ and $q$ are ME for $\tau$. 
As usual, we denote by $A^\Omega/\sim$ the quotient set of $A^\Omega$ by $\sim$, i.e., 
the set of all ME-equivalence classes of patterns supported by 
$\Omega$.
\par
In the proof of Theorem \ref{t:GOE-zd} we shall make use of the following elementary result
(here one should think of $a:=|A|$ as to the cardinality of the alphabet set and $a^{(n k )^d}$
(resp.\ $a^{(n k - 2r)^d}$, for $n k \geq 2r$) as to the number of all patterns supported by $\Omega_{n k}$
(resp.\ $\Omega_{n k}^-$), cf.\ \eqref{e:size-cubes}.

\begin{lemma}
Let $a,k,d,r$ be positive integers with $a \geq 2$. 
Then there exists $n_0 = n_0(a,k,d,r) \in \N$ such that
\begin{equation}
\label{e:akdn}
 \left(a^{k^d} - 1\right)^{n^d} < a^{(n k - 2r)^d}
\end{equation}
for all $n \geq n_0$.
\end{lemma}

\begin{proof}
Taking logarithms to base $a$, Inequality~\eqref{e:akdn} is equivalent to
\begin{equation}
\label{e:akdn2}
\log_a\left(a^{k^d} - 1\right) < \left(k - \frac{2r}{n}\right)^d.
\end{equation}
Since
\[
\log_a\left(a^{k^d} - 1\right) < \log_a\left(a^{k^d}\right) = k^d = \lim_{n \to \infty} \left(k - \frac{2r}{n}\right)^d,
\]
we deduce that there exists $n_0 \in \N$ such that \eqref{e:akdn2} and therefore \eqref{e:akdn} are satisfied for all $n \geq n_0$.
\end{proof}

\begin{proof}[Proof of Theorem~\ref{t:GOE-zd}]
We can assume that $a \coloneqq |A| \geq 2$.
\par
Suppose first that $\tau$ is not pre-injective.
Then, by Proposition~\ref{p:MEP-preinjectivity}, $\tau$ admits two distinct ME patterns,
say $p_1$ and $p_2$.
Denote by $\Omega \subset \Z^d$ their common support.
Since, for all $t \in \Z^d$, the patterns $t p_1$ and $t p_2$ (with support $t + \Omega$) are also distinct and ME, and any finite subset of $\Z^d$ containing the support of two distinct ME patterns is itself the support of two distinct ME patterns, we may assume that $\Omega = \Omega_k$ 
for some integer $k \geq 2 r$.
As the patterns $t p_1$ and $t p_2$ are ME, we have
\begin{equation}
\label{e:card-me-tplus}
| A^{t+\Omega_k}/\sim | \leq \vert A^{t+\Omega_k} \vert - 1 = a^{k^d}-1
\end{equation}
for all $t \in \Z^d$.
\par
Now observe that two patterns with support  
$\Omega_{n k}$ are ME if  their restrictions to
$t+\Omega_k$ are ME for every  $t \in T_n^k$.
Using~\eqref{e:card-me-tplus}, we deduce that
\[
| A^{\Omega_{n k}}/\sim| \leq  \prod_{t \in T_n^k} |A^{t + \Omega_k}/\sim| \leq  (a^{k^d}-1)^{n^d}.
\]
Taking $n \geq n_0(a,k,d,r)$, we then get
\[
\left|\tau(A^{\Z^d})_{\Omega_{n k}^-} \right|  \leq | A^{\Omega_{n k}} / \sim |  
\leq  (a^{k^d}-1)^{n^d} < a^{(n k - 2r)^d} = | A^{\Omega_{n k}^-} |.
\]
This implies that $\tau(A^{\Z^d})_{\Omega_{n k}^-} \subsetneqq A^{\Omega_{n k}^-}$, 
so that there must  exist a GOE pattern for $\tau$ with support $\Omega_{n k}^-$.
Consequently, $\tau$ is not surjective.
This shows that   (a) $\implies$ (b).
\par
Let us now turn to the proof of the converse implication.
Suppose that $\tau$ is not surjective. 
Then, by Proposition~\ref{p:surj-equiv-no-goe}, there exists a GOE pattern $p$ for $\tau$.
Since $tp$ is a GOE pattern for every $t \in \Z^d$, and any finite subset of $\Z^d$ containing the support of a GOE pattern is itself the support of a GOE pattern, we can assume that $p$ is supported by the cube $\Omega_k$ for some integer   $k \geq 2r$.
\par
Decompose again $\Omega_{n k}$ into the $n^d$ translates $t+\Omega_k$, with $t \in T_n^k$, and observe that
$tp \in A^{t+\Omega_k}$ is GOE for every $t \in T_n^k$.  Any pattern $q \in A^{\Omega_{n k}}$ which is not GOE satisfies that
$q\vert_{t+\Omega_k}$ is not GOE for every $t \in T_n^k$. As a consequence, we have
that
\begin{equation}
\label{e:g-n-1}
\left| \tau(A^{\Z^d})_{\Omega_{n k}} \right| 
\leq   \prod_{t \in T_n^k} \left| \tau(A^{\Z^d})_{t + \Omega_k} \right| 
\leq   (a^{k^d} -1)^{n^d}.
\end{equation}
\par
Let us fix now some element $a_0 \in A$ and consider the set $X$ consisting of all configurations 
$x \in A^{\Z^d}$ that satisfy
\[
x(g) = a_0 \quad \text{for all  } g \in \Z^d \setminus \Omega_{n k}^-.
\]
Observe that if $x_1$ and $x_2$ are in   $X$, then $\tau(x_1)$ and $\tau(x_2)$ coincide on 
$\Z^d \setminus \Omega_{n k}$.
It follows that
\[
|\tau(X)| = \left| \tau(X)_{\Omega_{n k}} \right|.
\]
On the other hand, taking $n \geq n_0(a,k,d,r)$ and using~\eqref{e:g-n-1}, 
we get
\[
\left| \tau(X)_{\Omega_{n k}} \right| 
\leq \left| \tau(A^{\Z^d})_{\Omega_{n k}} \right| 
\leq (a^{k^d} -1)^{n^d}
< a^{(n k - 2r)^d} = | A^{\Omega_{kn}^-} | = |X|
\]
and hence
\[
|\tau(X)| < |X|.
\]
By the pigeon-hole principle, this implies  that there exist two distinct configurations $x_1,x_2 \in X$ such that $\tau(x_1) = \tau(x_2)$.
As all configurations in $X$ are almost equal,
we deduce that   $\tau$ is not pre-injective.
This shows that (b) $\implies$ (a).
\end{proof}

\begin{remark}
The proof of the implication (a) $\implies$ (b) shows that if $\tau$ admits two distinct ME patterns supported by a
cube of side $k \geq 2 r$, then a cube of side $n k -2r$, with $n \geq n_0(a,k,d,r)$, must support a GOE pattern.
Conversely, a small addition to the proof of the implication (b) $\implies$ (a) yields that if $\tau$ admits a GOE pattern supported by a cube of side $k \geq 2 r$, then a cube of side $kn+2r$, with $n \geq n_0(a,k,d,r)$, supports two distinct ME patterns.
Indeed, the proof shows the existence of two configurations
$x_1,x_2 \in A^{\Z^d}$ that coincide outside of  $\Omega_{n k}^-$
and satisfy  $\tau(x_1) = \tau(x_2)$. It then follows from the proof of the implication (b) $\implies$ (a) in
Proposition~\ref{p:MEP-preinjectivity}  that the set 
\[
(\Omega_{n k}^- +(- S))+S = \Omega_{n k} + S = \Omega_{n k}^+
\] 
supports two distinct ME patterns.
\end{remark}

\section{The Garden of Eden theorem for general amenable groups}
\label{sec:goe-amenable}

\subsection{Amenability}
\label{ss:amenable-groups} 
(cf.~\cite{greenleaf}, \cite{paterson}, \cite[Chapter~4]{book}, \cite[Chapter~9]{coo-book})

\begin{definition}
A countable group $G$ is called \emph{amenable}\index{amenable group}\index{group!amenable ---} if there exists a sequence $(F_n)_{n \in \N}$ of nonempty finite subsets of $G$ such that
\begin{equation}
\label{e:folner-s}
\lim_{n \to \infty} \frac{\vert F_n \setminus  F_n g\vert}{\vert F_n \vert} = 0 \text{  for all } g \in G.
\end{equation}
 Such a sequence is called a \emph{F\o lner sequence}\index{F\o lner sequence} for $G$.
 \end{definition}

Note that if $A$ and $B$ are finite sets with the same cardinality, then 
$|A \setminus B| = |B \setminus A|$ and  $|A \bigtriangleup B| = |A \setminus B| + |B \setminus A| = 2 |A \setminus B|$,
where $\bigtriangleup$ denotes symmetric difference of sets.
As $| F g| = |F|$ for every finite subset $F \subset G$ and any $g \in G$,
it follows that Condition~\eqref{e:folner-s} is equivalent to each of the following conditions:
\begin{equation}
\label{e:folner-s-2}
\lim_{n \to \infty} \frac{\vert  F_n g \setminus  F_n \vert}{\vert F_n \vert} = 0 \text{  for all } g \in G,
\end{equation}
or
\begin{equation}
\label{e:folner-s-3}
\lim_{n \to \infty} \frac{\vert  F_n \bigtriangleup  F_n g \vert}{\vert F_n \vert} = 0 \text{  for all } g \in G.
\end{equation}

\begin{example}
All finite groups are amenable. Indeed, if $G$ is a finite group, then the constant sequence, 
defined by $F_n\coloneqq G$ for all $n \in \N$,  is  a F\o lner sequence for $G$ since $F_n \setminus  F_n g = \varnothing$ 
for every $g \in G$.
\end{example}

\begin{example}
The free abelian groups of finite rank $\Z^d$, $d \geq 1$,  are also amenable.
As a F\o lner sequence for $\Z^d$, one can take for instance the sequence of cubes 
\begin{equation}
\label{e:folner-for-Z-d}
F_n \coloneqq \{x \in \Z^d: \Vert x \Vert_\infty \leq n\} = \{0,\pm 1,  \dots,\pm n\}^d,
\end{equation}
where $\Vert x \Vert_\infty \coloneqq \max_{1 \leq i \leq d} |x_i|$ for all $x = (x_1,\dots,x_d) \in \Z^d$ is the sup-norm.
To see this, observe that
\begin{equation}
\label{e:ub-card-fn-Zd}
|F_n| = (2n + 1)^d
\end{equation} 
and, by the triangle inequality, 
$$
F_n + g \subset F_{n + \Vert g \Vert_\infty} \text{  for all } g \in \Z^d.
$$
Since
$$
F_n \subset F_{n + \Vert g \Vert_\infty},
$$
this implies
\begin{equation}
\label{e:up-bound-Fn-Zd}
|( F_n + g) \setminus F_n| \leq  (2n + \Vert g \Vert_\infty + 1)^d - (2n + 1)^d.
\end{equation}
As the right-hand side of~\eqref{e:up-bound-Fn-Zd} is a polynomial of degree $d - 1$ in $n$
while $|F_n|$ is a polynomial of degree $d$ in $n$ by~\eqref{e:ub-card-fn-Zd}, we conclude that
\begin{equation*}
\lim_{n \to \infty} \frac{\vert ( F_n + g ) \setminus  F_n \vert}{\vert F_n \vert} = 0, 
\end{equation*}
which is~\eqref{e:folner-s-2} in additive notation.
\end{example}

Let $G$ be a finitely generated group.
If $S \subset G$ is a finite symmetric generating subset, the \emph{Cayley graph}\index{Cayley graph} of $G$ with respect to $S$ 
is the graph $\GG(G,S)$ whose set of vertices  is $G$
and two vertices $g,h \in G$ are joined by an edge if and only if $h = g s$ 
for some $s \in S$.
Equip the set of vertices of $\GG(G,S)$   with its graph metric and consider  
the ball $B_n \subset G$ of radius $n$ centered at $1_G$.
It is easy to see that the sequence of positive integers $(|B_n|)_{n \in  \N}$ is submultiplicative.
Thus the limit
\begin{equation}
\label{e:growth}
\gamma(G,S) \coloneqq \lim_{n \to \infty} \sqrt[n]{|B_n|} 
\end{equation}
exists and satisfies $1 \leq \gamma(G,S) < \infty$.
One says that the group $G$ has \emph{subexponential growth}\index{subexponential growth}\index{group!--- of subexponential growth} if
$\gamma(G,S) = 1$ and  \emph{exponential growth}\index{exponential growth}\index{group!--- of exponential growth}  if $\gamma(G,S) > 1$. 
The fact that $G$ has subexponential (resp.~exponential) growth does not depend on the choice of the finite generating subset $S \subset G$ although the value of $\gamma(G,S)$ does.

\begin{example} 
The groups $\Z^d$ have subexponential growth.
Indeed, if $e_1,\ldots,e_d$ is the canonical basis of $\Z^d$,
and we take $S \coloneqq \{\pm e_1, \dots, \pm e_d\}$, then
the graph distance between two vertices $g$ and $h$ of $\GG(\Z^d,S)$ is 
$\Vert g - h \Vert_1$, where we write
$\Vert x \Vert_1 \coloneqq \sum_{1 \leq i \leq d} |x_i|$ for the $1$-norm of $x = (x_1,\dots,x_d) \in \Z^d$.
 We then have
$$
B_n =  \{g \in \Z^d : \Vert g \Vert_1 \leq n\} \subset \{0,\pm 1, \dots,\pm n\}^d
$$
and hence
$|B_n| \leq  (2 n + 1)^d$.
This implies that $\Z^d$ has subexponential growth.
Here it can be checked that the sequence  $(B_n)_{n \in \N}$ is also a F\o lner sequence for $\Z^d$.
\end{example}

When $G$ is an arbitrary  finitely generated group of   
subexponential growth and $S$ is a finite symmetric generating set for $G$, 
  it can be shown that one can always extract   a F\o lner sequence from the sequence 
$(B_n)_{n \in \N}$.  
Consequently, every finitely generated group with subexponential growth is amenable.

\begin{example}
\label{ex:F2-nonamen}
A (nonabelian) free group on two generators has exponential growth and is  not amenable.
Indeed, let $G$ be a free group based on two generators $a$ and $b$.
Consider the finite symmetric generating subset $S \subset G$ defined by
$$
S \coloneqq  \{a,b,a^{-1},b^{-1}\}.
$$
Then every element $g \in G$ can be uniquely written in \emph{reduced form},\index{reduced form of an element of a free group} i.e., in the form
\begin{equation}
\label{e:normal-form-in-free}
g = s_1 s_2 \dots s_n,
\end{equation}
where $n \geq 0$,  $s_i \in S$ for all $1 \leq i \leq n$, and $s_{i + 1} \not= s_i^{-1}$ for all $1 \leq i \leq n - 1$.
The integer $\ell_S(g) \coloneqq n$ is called the \emph{length}\index{length!--- of an element in a free group} of $g$ with respect to the generators $a$ and $b$.  
It is equal to the distance from $g$ to $1_G$ in the Cayley graph $\GG(G,S)$.
We deduce that  $|B_n| = 4 \cdot 3^{n - 1}$ for all $n \geq 1$ so that
  \[
  \gamma(G,S) = \lim_{n \to \infty} \sqrt[n]{4 \cdot 3^{n - 1}} =  3 > 1.
  \]
This shows that $G$ has exponential growth.  
\par
 Now suppose by contradiction that 
$(F_n)_{n \in \N}$ is a F\o lner sequence for $G$
and choose  some positive real number $\varepsilon <  1/2$.
Since the sequence $(F_n)_{n \in \N}$ is F\o lner, 
it follows from~\eqref{e:folner-s} that there exists  an integer $N \geq 0$ such that the set 
$F \coloneqq F_{N}$ satisfies  
\begin{equation}
\label{e:impossible-in-free}
|F \setminus F s| \leq \varepsilon |F|  \quad \text{for all  }s \in S.
\end{equation}
Denote,
for each $s \in S$, 
by  $G_s $ the subset of $G$ consisting of all elements
 $g \not= 1_G$ 
whose reduced form   ends with the letter 
$ s^{-1}$.
The four   sets $G_s$, $s \in S$, are pairwise disjoint
so that
\begin{equation}
\label{e:Gs-disjoin-F}
\sum_{s \in S} |F \cap G_s| \leq |F|.
\end{equation}
On the other hand, for each $s \in S$, we have that
\begin{equation}
\label{e:s-F-setminus-G-s}
|F| = |F \setminus G_s| + |F \cap G_s| =  |(F \setminus G_s) s| + |F \cap G_s|.
\end{equation}
 We now observe that  
$$
(G \setminus G_s) s \subset G_{s^{-1}}
$$
so that
$$
(F \setminus G_s) s \subset (F s \setminus F) \cup (F \cap G_{s^{-1}}) 
$$
and hence
\begin{align*}
|(F \setminus G_s) s| &\leq |F s \setminus F| + |F \cap G_{s^{-1}}|  \\ 
&= |F \setminus  F s| +  |F \cap G_{s^{-1}}| \\
&\leq \varepsilon |F| +  |F \cap G_{s^{-1}}| && \text{(by \eqref{e:impossible-in-free})}.  
\end{align*}
By using   \eqref{e:s-F-setminus-G-s}, we deduce that
$$
|F| \leq \varepsilon |F| +  |F \cap G_{s^{-1}}| + |F \cap G_s| 
$$
for all $s \in S$.
After summing up over all $s \in S$, this yields
\begin{align*}
4|F| &\leq 4\varepsilon |F| +  \sum_{s \in S} \left( |F \cap G_{s^{-1}}| + |F \cap G_s| \right) \\
&= 4\varepsilon |F| +  2\sum_{s \in S} |F \cap G_s|.   
\end{align*}
Finally, combining with \eqref{e:Gs-disjoin-F},   we  obtain
$$
4|F| \leq 4 \varepsilon |F| + 2 |F|
$$
and hence $|F| \leq 2 \varepsilon |F|$,  which is a contradiction since $F \not= \varnothing$ and  $\varepsilon < 1/2$.
 This proves that $G$ is not amenable.  
\end{example}

The class of amenable groups is closed under the operations of taking subgroups,
quotients, extensions (this means that if $1 \to H \to G \to K \to 1$ is an exact sequence with both $H$ and $K$  amenable, so is $G$), and inductive limits.
Consequently, all locally finite groups, all abelian groups and, more generally, all solvable groups are amenable. On the other hand every group
containing a free subgroup on two generators  is nonamenable. 
This implies for instance that all 
nonabelian free groups and the 
groups $\SL(n,\Z)$, $n \geq 2$, are nonamenable.
However, there are  groups containing no free subgroups on two generators that are nonamenable.
The first examples of such a group was given in~\cite{olshanskii-monster}
where Ol'{\v{s}}anski{\u\i}\index{Ol'{\v{s}}anski{\u\i}, Alexander Yu.} constructed a nonamenable \emph{monster group}\index{monster group}\index{group!monster ---} in which every proper subgroup is cyclic.
\par
Let us note  that there are finitely generated groups of exponential growth that are amenable.
For example, the  Baumslag-Solitar group\index{Baumslag-Solitar group}\index{group!Baumslag-Solitar ---} 
$BS(1,2)$, i.e.,  the group with presentation  $\langle a,b : a b a^{-1} = b^2 \rangle$, and the 
lamplighter group,\index{lamplighter group}\index{group!lamplighter ---}
i.e., the wreath product $(\Z/2\Z)\wr \Z$, 
have exponential growth but  are both solvable and hence amenable 
(cf.~\cite{de-la-harpe}).
\par
The original definition of amenability that was given by von Neumann~\cite{von-neumann}\index{von Neumann, John}  in 1929 is that a group $G$ is amenable if there exists
a finitely additive
invariant probability measure defined on the set of all subsets of $G$.
A key observation due to Day~\cite{day-ijm-1957}\index{Day, Mahlon M.} is that this is equivalent to the existence of  an invariant mean on the Banach space $\ell^{\infty}(G)$ of bounded real-valued functions on $G$.
 It is also in~\cite{day-ijm-1957} that the term \emph{amenable}\index{amenable group} occured for the first time
 (see \cite[p.~137]{ornstein-weiss}).
The fact that nonabelian free groups are not amenable is related to  the Hausdorff-Banach-Tarski paradox\index{Hausdorff-Banach-Tarski paradox} which actually was the motivation of von Neumann for introducing the notion of amenability.

\subsection{Entropy}
Let $G$ be  countable group, $A$ a finite set, and $X$ a subset
of $A^G$ (not necessarily a subshift).
Given a finite subset $\Omega \subset G$, recall (cf.~\eqref{e:x-VV}) that
$$
X_\Omega \coloneqq  \{x\vert_\Omega: x \in X\} \subset A^\Omega.
$$
Suppose now that the group $G$ is amenable and fix a F\o lner sequence $\FF = (F_n)_{n \in \N}$ for $G$.
The \emph{entropy}\index{entropy}\index{subshift!entropy of a ---} of $X$ (with respect to $\FF$) is defined by
\begin{equation}
\label{e:entropy}
h_\FF(X) \coloneqq \limsup_{n \to \infty} \frac{\log |X_{F_n}|}{|F_n|}.
\end{equation}
Since $X_F \subset A^F$ and hence $\log |X_F| \leq |F| \cdot \log |A|$ for every finite subset $F \subset G$, we always have
$$
h_\FF(X) \leq h_\FF(A^G) = \log |A|.
$$ 

\begin{example}
Take $G = \Z$ and  $A = \{0,1\}$.
Let us compute the entropy of the golden mean subshift $X \subset A^G$\index{entropy!--- of the golden mean subshift}\index{golden mean subshift!entropy of the ---}\index{subshift!golden mean ---} 
(cf.~Example~\ref{ex:golden-sub}) with respect to the F\o lner sequence
$\FF = (F_n)_{n \in \N}$, where $F_n \coloneqq \{0,1,\dots,n\}$.
We observe that   $u_n \coloneqq |X_{F_n}|$ satisfies $u_0 = 2$, $u_1 = 3$, and
$u_{n + 2} = u_{n + 1} + u_n$ for all $n \geq 2$. Thus, the sequence $(u_n)_{n \in \N}$ is a Fibonacci sequence and, by Binet's formula,
$$
u_n = \frac{1}{\sqrt{5}} \left(\varphi^{n + 3} - (1 - \varphi)^{n + 3}\right),
$$
where $\varphi \coloneqq  (1 + \sqrt{5})/2$ is the golden mean (this is the origin of  the name of this subshift). It follows that
$$
h_\FF(X) = \limsup_{n \to \infty} \frac{\log u_n}{n + 1} = \log \varphi.
$$
\end{example}

\begin{example}
Take again $G = \Z$ and  $A = \{0,1\}$,
and let us compute now the entropy of the even subshift $X \subset A^G$\index{entropy!--- of the even subshift}\index{even subshift!entropy of the ---}\index{subshift!even ---}  
(cf.~Example~\ref{ex:even-subshift}) with respect to the F\o lner sequence
$\FF = (F_n)_{n \in \N}$, where $F_n \coloneqq \{0,1,\dots,n\}$.
We observe that   $u_n \coloneqq |X_{F_n}|$ satisfies $u_0 = 2$, $u_1 = 4$, and
$u_{n + 2} = 1 + u_{n + 1} + u_n $ for all $n \geq 2$. Thus, the sequence $(v_n)_{n \in \N}$,
defined by $v_n \coloneqq  1 + u_n$ for all $n \in \N$,  is a Fibonacci sequence.
As $v_0 = 3$ and $v_1 = 5$,  by applying again Binet's formula, we get
\[
u_n = - 1 + v_n = - 1 + \frac{1}{\sqrt{5}} \left(\varphi^{n + 5} - (1 - \varphi)^{n + 5}\right),
\]
where $\varphi$ is the golden mean.
 It follows that
\[
h_\FF(X) = \limsup_{n \to \infty} \frac{\log u_n}{n + 1} = \log \varphi.
\]
Thus, the even subshift has the same entropy as the golden mean subshift with respect to $\FF$.
\end{example}

\begin{example}
\label{ex:ent-ledrappier-sub}
Take $G = \Z^2$ and  $A = \Z/2\Z$.
Let us compute the entropy of the Ledrappier subshift\index{entropy!--- of the Ledrappier subshift}\index{Ledrappier subshift!entropy of the ---}\index{subshift!Ledrappier ---} 
$X \subset A^G$ (cf.~Example~\ref{ex:ledrappier-sub}) with respect to the F\o lner sequence
$\FF = (F_n)_{n \in \N}$, where $F_n \coloneqq \{0,1,\dots,n\}^2$.
We observe that, for each $x \in X$, the pattern $x\vert_{F_n}$ is entirely determined by 
$x \vert_{H_n}$, where $H_n \subset \Z^2$ is the horizontal interval
$H_n \coloneqq \{0,1, \dots,2n\} \times \{0\}$. 
Therefore,    $u_n \coloneqq |X_{F_n}|$ satisfies
$$
\log u_n \leq |H_n| \cdot \log|A| = (2n + 1) \log 2.
$$
This gives us
$$
h_\FF(X) = \limsup_{n \to \infty} \frac{\log u_n}{(n + 1)^2} = 0.
$$
\end{example}

\begin{remark}
It is a deep result due to Ornstein\index{Ornstein, Donald S.} and Weiss~\cite{ornstein-weiss}\index{Weiss, Benjamin} that,
when $G$ is a countable amenable group, $A$ a finite set, and $X \subset A^G$  a subshift, 
then the $\limsup$ in \eqref{e:entropy} is actually a true limit and does not depend on the
particular choice of the F\o lner sequence $\FF$ for $G$.
However, we shall not need it in the sequel.
\end{remark}

An important property of cellular automata that we shall use in the proof of the Garden of Eden theorem  below is that they cannot increase  entropy.
More precisely, we have the following result.

\begin{proposition}
\label{p:entropie-image}
Let $G$ be a countable amenable group with F\o lner sequence $\FF = (F_n)_{n \in \N}$ and
$A,B$  finite sets. 
Suppose that  $\tau \colon A^G \to B^G$ 
is a cellular automaton and $X$ is a subset of $A^G$. 
Then one has $h_\FF(\tau(X)) \leq h_\FF(X)$.
\end{proposition}

For the proof, we shall use the following general property of F\o lner sequences.

\begin{lemma}
\label{l:folner-riht-fset}
Let $G$ be a countable amenable group with F\o lner sequence $\FF = (F_n)_{n \in \N}$ and
let $S$ be a finite subset of $G$.
Then one has
\begin{equation}
\label{e:folner-riht-fset}
\lim_{n \to \infty} \frac{|F_n S \setminus F_n|}{|F_n|} = 0.  
\end{equation}
\end{lemma}

\begin{proof}
Observe that
\[
F_n  S \setminus F_n = \bigcup_{s \in S} (F_n  s \setminus F_n)
\]
so that
\[
|F_n  S \setminus F_n| = \left| \bigcup_{s \in S} (F_n  s \setminus F_n)  \right| 
\leq \sum_{s \in S} |F_n  s \setminus F_n|.
\]
Thus we get
\[
\frac{|F_n S \setminus F_n|}{|F_n|}
\leq \sum_{s \in S} \frac{|F_n s \setminus F_n|}{|F_n|}
\]
for all $n \in \N$. As 
\[
\lim_{n \to \infty} \frac{|F_n s \setminus F_n|}{|F_n|} = 0  
\]
for each $s \in S$ by~\eqref{e:folner-s-2}, this gives us~\eqref{e:folner-riht-fset}. 
\end{proof}

\begin{proof}[Proof of Proposition~\ref{p:entropie-image}]
Let us set $Y \coloneqq \tau(X)$ and let $S \subset G$ be a memory set for $\tau$
with $1_G \in S$.
Recall that it immediately follows from~\eqref{def:automate}  that if two configurations coincide on 
$g S$ for  some $g \in G$  then their images by $\tau$ take the same value at $g$.
We deduce  that
\begin{equation} 
|Y_{\Omega}| \leq |X_{\Omega S}|
\end{equation}
 for every finite subset $\Omega \subset G$.
Now observe that
\[
X_{\Omega S} \subset X_\Omega \times X_{\Omega S \setminus \Omega} \subset X_\Omega \times A^{\Omega S \setminus \Omega},
\]
so that we get
\[
\log |Y_\Omega| \leq \log |X_\Omega| +    | \Omega S \setminus \Omega| \cdot \log |A|. 
\]
After replacing $\Omega$ by $F_n$  and dividing both sides by $|F_n|$, this inequality becomes   
 \begin{equation}
 \label{e:ineq-ent-im-Y-X}
 \frac{\log |Y_{F_n}|}{|F_n|}
 \leq  \frac{\log |X_{F_n}|}{|F_n|} +
  \frac{|F_n S \setminus F_n|}{|F_n|}  \cdot \log |A|.
 \end{equation}
As 
\[
\lim_{n \to \infty} \frac{|F_n S \setminus F_n|}{|F_n|} = 0  
\]
by Lemma~\ref{l:folner-riht-fset}, taking the limsup in~\eqref{e:ineq-ent-im-Y-X} finally gives the required inequality $h_\FF(Y) \leq h_\FF(X)$.
\end{proof}

\begin{corollary}
\label{c:amen-nosurj-ca-less-lettters}
Let $G$ be a countable amenable group and let
$A,B$ be  finite sets with $|A| < |B|$.
Then there exists no surjective cellular automaton $\tau \colon A^G \to B^G$. 
\end{corollary}

\begin{proof}
This is an immediate consequence of 
Proposition~\ref{p:entropie-image}
since
\[ 
h_\FF(A^G) = \log |A| < \log |B| = h_\FF(B^G).
\]
\end{proof}

The following example (cf.~\cite[p.~138]{ornstein-weiss})  shows that 
Corollary~\ref{c:amen-nosurj-ca-less-lettters} becomes false if  
 the amenability hypothesis is removed.
 
 \begin{example}
 Let $G$ be the free group on two generators $a$ and $b$.
 Take $A \coloneqq  \Z/2\Z$ and $B \coloneqq  \Z/2\Z \times \Z/2\Z$, so that $|A| = 2$ and $|B| = 4$.
 Consider the map $\tau \colon A^G \to B^G$ defined by
 \[
 \tau(x)(g) = (x(g) + x(g a), x(g) + x(g b))
 \]  
 for all $x \in A^G$ and $g \in G$.
 Observe that $\tau$ is a cellular automaton with memory set $S = \{1_G,a,b\}$ and local defining map
$ \mu \colon A^G \to B$ given by
 \[
 \mu(p) = (p(1_G) + p(a),p(1_G) + p(b))
 \]
  for all $p \in A^S$. It is easy to check that $\tau$ is surjective.
Note that $A^G$ and $B^G$ are totally disconnected compact abelian topological groups and $\tau$ is a continuous group morphism whose kernel consists of the two constant configurations in $A^G$.
 \end{example}

\subsection{Tilings}

Let $G$ be a group. 
Given a finite subset $E\subset G$, let us say that a subset $T \subset G$ is an $E$-\emph{tiling}\index{tiling} of $G$ provided the sets $t E$, $t \in T$, are pairwise disjoint and there exists a finite subset
$E' \subset G$ such that the sets    $t E'$, $t \in T$,  cover $G$.

\begin{example}
Take $G = \Z^d$ and $E = \{0, \pm 1, \pm 2, \ldots, \pm m\}^d$ for some $m \in \N$, then $T: = ((2m+1)\Z)^d  \subset \Z^d$ is an $E$-tiling
(here one can take $E' =  E$).
\end{example}
  
Given  any nonempty finite subset $E$ of a group $G$,
we can use  Zorn's lemma to prove that
there always exists  an $E$-tiling $T \subset G$.
Indeed, consider the set $\mathcal{S}(E)$ consisting of all subsets $S \subset G$ such that the sets $sE$, $s \in S$, are pairwise disjoint.
We first observe that $\mathcal{S}(E)$ 
is nonempty 
since $\{1_G\} \in \mathcal{S}(E)$.
On the other hand, $\mathcal{S}(E)$
is inductive  with respect to set inclusion since  
if ${\mathcal S}' \subset \mathcal{S}(E)$ is a chain, then $M:= \cup_{S \in {\mathcal S}'} S$ belongs to
$\mathcal{S}(E)$ and is an upper bound for ${\mathcal S}'$. By Zorn's lemma, there exists a maximal element $T \in {\mathcal S}(E)$. 
As $T \in {\mathcal S}(E)$,
 the sets $t E$, $t \in T$, are pairwise disjoint. 
 Now, given any  $g \in G$, we can find, by maximality of $T$,  an element $t = t(g) \in T$ such that 
 $g E \cap t E \neq \varnothing$ and hence
$g \in t E E^{-1} $.
It follows that the sets $t E E^{-1}$, $t \in T$, cover $G$. 
Since the set $E' \coloneqq E E^{-1}$ is finite, this shows that $T$ is an $E$-tiling of $G$.

For the proof of the Garden of Eden theorem  in the next subsection,
we shall use some technical results about tilings in amenable groups.

\begin{lemma}
Let $G$ be a countable amenable group with F\o lner sequence  $\FF = (F_n)_{n \in \N}$.
Let $E \subset G$ be a nonempty finite subset and
$T \subset G$ an $E$-tiling. 
Define, for  $n \in \N$, the subset $T_n \subset T$  by
 \[
 T_n \coloneqq \{t \in T: t E \subset F_n\}.
 \]
Then there exist   a constant 
$\alpha = \alpha(\FF,T) >0$ and $n_0 \in \N$ such that
\begin{equation}
\label{e:t-n}
|T_n| \geq \alpha |F_n| \ \ \mbox{ for all } n \geq n_0.
\end{equation}
\end{lemma}

\begin{proof}
Since $T$ is a $E$-tiling, there exists a finite subset $E' \subset G$ such that the sets $t E'$, $t \in T$, cover $G$.
After replacing $E'$ by $E' \cup E$, if necessary, we may assume that $E \subset E'$.  
Define, for $n \in \N$,
\[
T_n^+ \coloneqq  \{t \in T: tE' \cap F_n \neq \varnothing\}.
\]
Clearly $T_n \subset T_n^+$.
As the sets $tE'$, $t \in T_n^+$, cover $F_n$, we have
$|F_n| \leq |T_n^+| \cdot |E'|$ so that
\begin{equation}
\label{e:premiere-tiling}
\frac{|T_n^+|}{|F_n|} \geq \frac{1}{|E'|}
\end{equation}
for all $n \in \N$.
Now observe that
\[
T_n^+ = T \cap \left(\bigcup_{g \in E'} F_n g^{-1} \right)
\text{  and  }
T_n = T \cap \left(\bigcap_{h \in E}  F_n h^{-1} \right), 
\]
so that
\begin{align*}
T_n^+ \setminus T_n
&= T \cap \left( \bigcup_{g \in E'} F_n g^{-1}  \setminus \bigcap_{h \in E}  F_n h^{-1}\right)  \\
&\subset   \bigcup_{g \in E'} F_n g^{-1}  \setminus \bigcap_{h \in E}  F_n h^{-1} \\
&= \bigcup_{g \in E', \ h \in E} (F_n g^{-1} \setminus F_n h^{-1}).
\end{align*} \\
We deduce that
\[
|T_n^+ \setminus T_n|
\leq \sum_{g \in E', \ h \in E} |F_n g^{-1}  \setminus F_n h^{-1}|
= \sum_{g \in E', \ h \in E} |F_n   \setminus F_n h^{-1} g|.
\]
As
\[
\lim_{n \to \infty} \frac{|F_n  \setminus F_n h^{-1} g|}{|F_n|} =0
\]
 for all $g \in E'$ and $h \in E$ by~\eqref{e:folner-s},  
 it follows that
\[
\frac{|T_n^+ | - \vert T_n \vert}{|F_n|} = \frac{\vert T_n^+ \setminus T_n \vert}{|F_n|} \to 0
\]
as  $n \to \infty$. 
Using~\eqref{e:premiere-tiling} and taking $\varepsilon \coloneqq \dfrac{1}{2|E'|}$, 
we deduce that  there exists $n_0 \in \N$ such that
\[
\frac{|T_n|}{|F_n|} = \frac{|T_n^+|}{|F_n|} - \frac{\vert T_n^+ \vert - \vert T_n \vert}{|F_n|} \geq  \frac{1}{|E'|} - \varepsilon = \alpha,
\]
where $\alpha \coloneqq \dfrac{1}{2|E'|}$, for all $n \geq n_0$. 
\end{proof}

\begin{proposition}
\label{p:inegalite-entropique-tiling}
Let $G$ be a countable amenable group with F\o lner sequence  $\FF = (F_n)_{n \in \N}$
and let $A$ be a finite set.  
Let $X \subset A^G$ be a subset and suppose there exist a nonempty finite subset $E \subset G$
and an $E$-tiling $T \subset G$ such that $X_{tE} \subsetneqq A^{tE}$ for all $t \in T$. Then $h_\FF(X) < \log |A|$.
\end{proposition}

\begin{proof}
Let us set, as above, $T_n \coloneqq   \{t \in T: tE \subset F_n\}$ and write 
\[
F_n^* \coloneqq F_n \setminus \bigcup_{t \in T_n} t E,
\] 
for all $n \in \N$.
Observe that $\bigcup_{t \in T_n} t E \subset F_n$ so that 
$$
X_{F_n} \subset A^{F_n^*} \times \prod_{t \in T_n} X_{tE}
$$ 
and
\begin{equation}
\label{e:tiling-Fn-star-less}
|F_n| = |F_n^*| + |T_n| \cdot |E|.
\end{equation}
 It follows that
\[
\begin{split}
\log |X_{F_n}| & \leq  \vert F_n^*\vert \cdot \log |A| +  \sum_{t \in T_n} \log |X_{tE}| \\
& \leq  |F_n^*| \cdot \log |A| +  \sum_{t \in T_n}  \log \left(|A^{tE}|-1\right) \\
& = \vert F_n^* \vert \cdot  \log |A| +  |T_n| \cdot  \log (|A|^{|E|} - 1) \\
& = \vert F_n^* \vert \cdot \log |A| +  |T_n| \cdot |E| \cdot \log |A| + |T_n| \cdot \log (1 - |A|^{-|E|})\\
& = \vert F_n \vert \cdot \log |A| + |T_n| \cdot \log (1 - |A|^{-|E|}),
\end{split}
\]
where the last equality follows from~\eqref{e:tiling-Fn-star-less}.
Setting $c\coloneqq - \log (1 - |A|^{-|E|}) > 0$,  
we deduce that
\[
h_\FF(X) = \limsup_{n \to \infty} \frac{\log |X_{F_n}|}{|F_n|}  \leq \log |A| - c \alpha < \log |A|,
\]
where $\alpha = \alpha(\FF,T)$ is as in~\eqref{e:t-n}.
\end{proof}

\begin{corollary}
\label{c:ineq-entropy-subshifts}
Let $G$ be a countable amenable group with F\o lner sequence  $\FF = (F_n)_{n \in \N}$
and let $A$ be a finite set.  
Let $X \subset A^G$ be a subshift 
and suppose that there exists a nonempty finite subset $E \subset G$
such that $X_{E} \subsetneqq A^{E}$.  Then one has $h_\FF(X) < \log |A|$.
\end{corollary}

\begin{proof}
If $T$ is an $E$-tiling of $G$,
we deduce from the shift-invariance of $X$
that
$X_{t E} \subsetneqq A^{t E}$ for all $t \in T$, so that Proposition~\ref{p:inegalite-entropique-tiling} applies. 
\end{proof}

\subsection{The Garden of Eden theorem for amenable groups}
The following result is due to Mach{\`{\i}},  Scarabotti, and the first author~\cite{ceccherini}.
Since the groups $\Z^d$ are all amenable,
it extends Theorem~\ref{t:GOE-zd}. 

\begin{theorem}
\label{t:GOE-amen}\index{Garden of Eden theorem!--- (for amenable groups)}\index{theorem!Garden of Eden --- (for amenable groups)}
Let $G$ be a countable amenable group with F\o lner sequence   
$\FF = (F_n)_{n \in \N}$ and $A$ a finite set.
Suppose that $\tau \colon A^G \to A^G$ is a cellular automaton. Then the following conditions are
equivalent:
\begin{enumerate}[{\rm (a)}]
\item $\tau$ is surjective;
\item $h_\FF(\tau(A^G)) = \log |A|$;
\item $\tau$ is pre-injective.
\end{enumerate}
\end{theorem}

\begin{proof}
The implication (a) $\implies$ (b) is obvious  since $h_\FF(A^G) = \log|A|$.
\par
In order to show the converse implication, let us suppose that $\tau$ is not surjective, that is,
the image subshift $X \coloneqq \tau(A^G)$ is such that  $X \subsetneqq A^G$.
Since $X$ is closed in $A^G$,
there exists  a finite subset $E \subset G$ such that
$X_E \subsetneqq A^E$.
By applying Corollary~\ref{c:ineq-entropy-subshifts},
we deduce that
$h_\FF(X)  < \log|A|$. This shows (b) $\implies$ (a).
\par
Let $S \subset G$ be a memory set for $\tau$ such that $1_G \in S$.
\par
Let us  show (b) $\implies$ (c). Suppose that $\tau$ is not pre-injective. By virtue of
Proposition~\ref{p:MEP-preinjectivity}, we can find
a nonempty finite subset $\Omega \subset G$ and two distinct patterns $p_1, p_2 \in A^\Omega$
that are mutually erasable for $\tau$.
Let $E \coloneqq  \Omega S^{-1}S$.
Observe that $\Omega \subset E$ since $1_G \in S$.
Let  $T \subset G$ be an $E$-tiling of $G$.
Consider the subset $Z \subset A^G$ defined by
\[
Z\coloneqq \{z \in A^{G}: z\vert_{t \Omega} \neq tp_1 \mbox{ for all } t \in T\}.
\]
Observe that $Z_{tE} \subsetneqq A^{t E}$ for all $t \in T$.
By using Proposition~\ref{p:entropie-image} and Proposition~\ref{p:inegalite-entropique-tiling}, we deduce that $h_\FF(\tau(Z))  \leq h_\FF(Z) < \log |A|$.
We claim that $\tau(Z) = \tau(A^G)$. Let $x \in A^G$. Let $T_x\coloneqq \{t \in T: x\vert_{t\Omega} = tp_1\}$ and define
$z \in Z$ by setting, for all $g \in G$,
\[
z(g) \coloneqq
\begin{cases}
tp_2(g) & \mbox{ if $g \in t\Omega$ for some $t \in T_x$}\\
x(g) & \mbox{ otherwise.}
\end{cases}
\]
Let us check that $\tau(z) = \tau(x)$. Let $g \in G$.
If $g \notin \cup_{t \in T_x} t\Omega S^{-1}$, then $gS \cap t\Omega = \varnothing$ for all $t \in T_x$ and therefore $z\vert_{gS} =
x\vert_{gS}$, so that $\tau(z)(g) = \tau(x)(g)$.
Suppose now that $g \in t\Omega S^{-1}$ for some (unique) $t = t(g) \in T_x$ and consider the
configuration $y \in A^G$ defined by setting, for all $h \in G$,
\[
y(h) \coloneqq
\begin{cases}
tp_2(h) & \mbox{ if } h \in t\Omega\\
x(h) & \mbox{ otherwise.}
\end{cases}
\]
Observe that $x\vert_{G \setminus t\Omega} = y\vert_{G \setminus t\Omega}$. Since the patterns $x\vert_{t\Omega} = tp_1$ and $y\vert_{t\Omega} = tp_2$ are mutually erasable, we deduce that $\tau(y) = \tau(x)$. Moreover, as $gS \subset  t\Omega S^{-1}S = tE$, we have $z\vert_{gS} = y\vert_{gS}$, and therefore
$\tau(z)(g) = \tau(y)(g) = \tau(x)(g)$.
This shows that $\tau(z) = \tau(x)$, and the claim follows.
We conclude that $h_\FF(\tau(A^G)) = h_\FF(\tau(Z)) < \log |A|$. This shows the implication (b) 
$\implies$ (c).
\par
Finally, let us show  (c) $\implies$ (b).
Let us set as above $X \coloneqq  \tau(A^G)$ and suppose that $h_\FF(X) < \log |A|$.
As $1_G \in S$, we have $F_n \subset F_n S^{-1}$ so that
$$
X_{F_n S^{-1}} \subset X_{F_n} \times A^{F_n S^{-1} \setminus F_n},
$$
for all $n \in \N$.
We deduce that
\begin{equation}
\label{e:inegalite-X-F-nS}
\frac{\log |X_{F_n S^{-1}}|}{|F_n|} \leq
\frac {\log |X_{F_n}|}{|F_n|}
+  \frac{|F_n S^{-1} \setminus F_n|}{|F_n|} \log |A|.
\end{equation}
As
$$
\lim_{n \to \infty} \frac{|F_n S^{-1} \setminus F_n|}{|F_n|} =0
$$
by Lemma~\ref{l:folner-riht-fset},
we deduce from~\eqref{e:inegalite-X-F-nS} that
\[
\limsup_{n \to \infty} \frac{\log |X_{F_n S^{-1}}|}{|F_n|} \leq
\limsup_{n \to \infty} \frac {\log |X_{F_n}|}{|F_n|}
= h_\FF(X) < \log |A|.
\]
Consequently, we can find $n_0 \in \N$ such that,
\begin{equation}
\label{e:X-F-nS-A}
|X_{F_{n_0} S^{-1}}| < |A|^{|F_{n_0}|}.
\end{equation}
Fix $a_0 \in A$ and consider the subset $Z \subset A^G$ defined by
\[
Z\coloneqq \{z \in A^G: z(g) = a_0 \mbox{ for all } g \in G \setminus F_{n_0}\}.
\]
Note that $|Z| = |A|^{|F_{n_0}|}$.
Let $z_1, z_2 \in Z$.
If $g \in G \setminus F_{n_0} S^{-1}$, then $z_1$ and $z_2$ coincide on $gS \subset G \setminus F_{n_0}$ so that $\tau(z_1)(g) = \tau(z_2)(g)$.
Therefore $\tau(z_1)$ and $\tau(z_2)$ coincide on $G \setminus F_{n_0}S^{-1}$.
This implies  that  $|\tau(Z)| \leq |X_{F_{n_0} S^{-1}}|$.
Using~\eqref{e:X-F-nS-A},
we deduce that $|\tau(Z)| < |Z|$.
By the pigeon-hole principle, there exist two distinct elements $z_1, z_2 \in Z$ such that
$\tau(z_1) = \tau(z_2)$.
As all elements in $Z$ are almost equal (they coincide outside of the finite set $F_{n_0}$),
we conclude that $\tau$ is not pre-injective.
\end{proof}

\section{Failure of the Garden of Eden theorem for nonamenable groups}
\label{sec:goe-fails-non-amen}

Let us say that a countable group $G$ has the \emph{Moore property}\index{Moore property!--- for a group}\index{group!--- satisfying the Moore property}  if 
every surjective cellular automaton $\tau \colon A^G \to A^G$ with finite alphabet $A$ over $G$ is 
pre-injective
and that it has the \emph{Myhill property}\index{Myhill property!--- for a group}\index{group!--- satisfying the Myhill property}  if
every pre-injective cellular automaton $\tau \colon A^G \to A^G$ with finite alphabet $A$ over $G$ is surjective.
Also let us  say that a countable group $G$ satisfies the \emph{Moore-Myhill property}\index{Moore-Myhill property!--- for a group}\index{group!--- satisfying the Moore-Myhill property} 
or that it satisfies the \emph{Garden of Eden theorem}\index{Garden of Eden theorem!group satisfying the ---}\index{group!--- satisfying the Garden of Eden theorem} 
  if $G$ has both the Moore and the Myhill properties.
Theorem~\ref{t:GOE-amen}
tells us that every countable amenable group satisfies the Garden of Eden theorem.
The examples below, essentially due to Muller~\cite{muller}\index{Muller, David E.}
(see also \cite[Section~6]{machi-mignosi}, \cite[Section~6]{ceccherini}, \cite[Chapter~5]{book}), 
show that neither the Moore nor the Myhill property holds for countable groups containing  nonabelian free subgroups.

\begin{example}
\label{e;ex1}
 Let $G$ be a countable group and suppose that $G$ contains two elements $a$ and $b$ generating a nonabelian free
subgroup $H \subset G$. 
Take $A = \{0,1\}$ and let $S \coloneqq \{a,a^{-1},b,b^{-1}\}$.  
 Consider the
cellular automaton $\tau \colon A^G \to A^G$ 
with memory set $\{1_G\} \cup S$
defined by
\[
\tau(x)(g) \coloneqq
\begin{cases}
0 & \text{ if } x(g) + x(ga) + x(ga^{-1}) + x(gb) + x(gb^{-1}) \leq 2 \\
1 & \text{ otherwise}
\end{cases}
\]
for all $x \in A^G$ and $g \in G$.
\par
The pair of configurations $(x_1, x_2) \in A^G \times A^G$, defined by $x_1(g) = 0$ for all $g \in G$, and $x_2(g) = 0$ for all $g \in G \setminus \{1_G\}$ and $x_2(1_G) = 1$, is a diamond for $\tau$. 
Therefore $\tau$ is not pre-injective.
However, $\tau$ is surjective.  To see this, let $y \in A^G$. Let us show that there exists $x \in A^G$ such that
$\tau(x) = y$. 
Let $R \subset G$ be a complete set of representatives of the left cosets of $H$ in $G$.
We define $x$ as follows.
Every element $g \in G$ can be uniquely written in the form $g = r h$ with $r \in R$ and $h \in H$. 
If $g \in R$, i.e., $h = 1_G$, we  set $x(g) \coloneqq 0$. 
Otherwise,
we set $x(g) \coloneqq y(rh^-)$,
where $h^-$ is the \emph{predecessor} of $h$ in $H$, i.e.,  the unique element  $h^- \in H$
such that $\ell_S(h^-) = \ell_S(h) - 1$ and $h = h^- s$ for some $s \in S$
(here $\ell_S(\cdot)$ denotes the length
of the reduced form for elements of $H$, see~Example~\ref{ex:F2-nonamen}).
One easily checks that  $\tau(x) = y$. This shows that $\tau$ is surjective.
Thus the Moore implication fails to hold for  groups containing nonabelian free subgroups.
\end{example}

\begin{example}
\label{e;ex2}
 Let $G$ be a countable group and suppose that $G$ contains two elements $a$ and $b$ generating a nonabelian free
subgroup $H \subset G$. 
Let $A = \Z/2\Z \times \Z/2\Z$ be the Klein four-group
 and consider the group endomorphisms $p$ and $q$ of $A$
respectively defined by $p(\alpha,\beta) = (\alpha, 0)$ and $q(\alpha,\beta) = (\beta, 0)$ for all  $(\alpha,\beta) \in A$.
Let $\tau \colon A^G \to A^G$ be the cellular automaton with memory set $S \coloneqq \{a,a^{-1}, b, b^{-1}\}$ defined by
\[
\tau(x)(g) = p(x(ga)) + p(x(ga^{-1}) + q(x(gb)) +  q(x(gb^{-1}))
\]
for all $x \in A^G$ and $g \in G$.
The image of $\tau$ is contained in $(\Z/2\Z \times \{0\})^G$. Therefore $\tau$ is not surjective.
We claim that $\tau$ is pre-injective. As $\tau$ is a group endomorphism of $A^G$, it suffices to show that there is no configuration with finite support in the kernel of $\tau$.
Assume on the contrary  that there is an element $x \in A^G$ with nonempty finite support 
$\Omega \coloneqq \{g \in G : x(g) \not= 0_A \} \subset G$ such that $\tau(x) = 0$.
Let $R \subset G$ be a complete set of representatives of the left cosets of $H$ in $G$.
Let us set 
$\Omega_r \coloneqq \Omega \cap rH$ for all $r \in R$.
Then $\Omega$ is the disjoint union of the sets $\Omega_r$, $r \in R$.
Let $r \in R$ such that $\Omega_r \neq \varnothing$
and consider an element $g = r h \in \Omega_r$ with $h \in H$ 
at maximal distance from the identity in the Cayley graph of $(H,S)$ (i.e., with $\ell_S(h)$ maximal).
We have $x(g) = (\alpha,\beta) \not= (0,0) = 0_A$.
Suppose first that $\alpha \not= 0$.
We can find $s \in \{a,a^{-1}\}$ such that $\ell_S(h s) = \ell_S(h) + 1$.
For all $t \in S \setminus \{s^{-1}\}$,
we have that  $\ell_S(h s t) = \ell(h) + 2$ and hence
$x(g s t) = 0_A$ by maximality.
It follows that
\[
\tau(x)(g s) = p(x(g ) )= (\alpha,0) \not= 0_A,
\]
which contradicts the fact that $x$ is in the kernel of $\tau$.
Suppose now that $\alpha = 0$.
Then $\beta \not= 0$.
We take now $s \in \{b,b^{-1}\}$ such that
$\ell_S(h s) = \ell_S(h) + 1$.
By an argument similar to the one that we used in the first case, we get
\[
\tau(x)(g s) = q(x(g ) )= (\beta,0) \not= 0_A,
\]
so that we  arrive at a contradiction also in this case.
Thus $\tau$ is pre-injective.
This shows that the Myhill implication fails to hold for  groups containing nonabelian free subgroups.
\end{example}

As mentioned in Subsection~\ref{ss:amenable-groups},
there are nonamenable countable groups containing no nonabelian free subgroups.
However, Bartholdi~\cite{bartholdi}\index{Bartholdi, Laurent} (see \cite[Chapter~5]{book}) proved that
the Moore property fails to hold for all nonamenable countable groups.
Recently, Bartholdi 
and Kielak~\cite{bartholdi:2016} also proved that the Myhill property fails to hold for all nonamenable countable groups.    
Combining these results with the Garden of Eden theorem for amenable groups 
(Theorem~\ref{t:GOE-amen}),
this yields the following characterization of amenability in terms of cellular automata. 

\begin{theorem}
Let $G$ be a countable group.
Then the following conditions are equivalent:
\begin{enumerate}[\rm (a)]
\item
$G$ is amenable;
\item
$G$ has the Moore property;
\item
$G$ has the Myhill property;
\item
$G$ satisfies the Garden of Eden theorem.
\end{enumerate}
\end{theorem}

\section{The Garden of Eden theorem for subshifts}
\label{sec:goe-subshifts}

\subsection{Strongly irreducible subshifts}
Let $G$ be a countable group and $A$ a finite set.
\par
A subshift $X \subset A^G$ is called \emph{strongly irreducible}\index{strongly irreducible subshift}\index{subshift!strongly irreducible ---} if there is a finite subset 
 $\Delta \subset G$ satisfying the following property:
 if $\Omega_1$ and $\Omega_2$   are finite subsets of $G$ such that  $\Omega_1 \Delta$ does not meet   $\Omega_2$,  then, given any two configurations $x_1,x_2 \in X$, there exists  a configuration $x \in X$ which coincides with $x_1$ on $\Omega_1$ and with $x_2$ on $\Omega_2$.

\begin{example}
The full shift $A^G$ is strongly irreducible (one can take $\Delta = \{1_G\}$).
\end{example}

\begin{example}
The even subshift $X \subset \{0,1\}^{\Z}$, described in Example~\ref{ex:even-subshift}, is strongly irreducible
(one can take $\Delta = \{-2,-1,0,1,2\}$).\index{even subshift}\index{subshift!even ---}
\end{example}

\begin{example}
The hard-ball model,  described in Example~\ref{ex:hard-ball}, is strongly irreducible
(one can take $\Delta = \{0, \pm e_1, \dots,\pm e_d\}$).\index{hard-ball model}\index{subshift!hard-ball model ---}
In particular  ($d = 1$), the golden mean subshift is strongly irreducible.
\end{example}

\begin{example}
The Ledrappier subshift, described in Example~\ref{ex:ledrappier-sub}, 
is not strongly irreducible.\index{Ledrappier subshift}\index{subshift!Ledrappier ---}
\end{example}

Fiorenzi~\cite[Theorem~4.7]{fiorenzi-strongly} obtained the following extension of Theorem~\ref{t:GOE-amen}.

\begin{theorem}
\label{t:fiorenzi-strongly}
Let $G$ be a countable amenable group with F\o lner sequence   
$\FF = (F_n)_{n \in \N}$ and $A,B$  finite sets.
Suppose that $X \subset A^G$ is a strongly irreducible subshift of finite type and $Y \subset B^G$ is a strongly irreducible subshift   with $h_\FF(X) = h_\FF(Y)$ and
that $\tau \colon X \to Y$ is a cellular automaton. Then the following conditions are
equivalent:
\begin{enumerate}[{\rm (a)}]
\item $\tau$ is surjective;
\item 
$h_\FF(\tau(X)) = h_\FF(Y)$;
\item $\tau$ is pre-injective.
\end{enumerate}
\end{theorem}

\begin{example}
The cellular automaton $\tau \colon X \to Y$ from the golden mean subshift to the even subshift described in Example~\ref{ex:ca-golden-even}
satisfies all the hypotheses in the previous theorem.
As $\tau$ is pre-injective (cf.~Example~\ref{ex:ca-golden-even-pre-inj}), we deduce that
$\tau$ is surjective.
Note that here one might also easily obtain  surjectivity of $\tau$ by a direct argument. 
\end{example}

\subsection{The Moore and the Myhill properties for subshifts}
\label{subsec:moore-myhill-subshifts}
Let $G$ be a countable group, $A$ a finite set, and $X \subset A^G$ a subshift.
One says that the subshift  $X$ has the \emph{Moore property}\index{Moore property!--- for a subshift}\index{subshift!--- satisfying the
Moore property}
if every surjective cellular automaton 
$\tau \colon X \to X$ is pre-injective
and that it has the \emph{Myhill property}\index{Myhill property!--- for a subshift}\index{subshift!--- satisfying the
Myhill property}  
if every pre-injective cellular automaton 
$\tau \colon X \to X$ is surjective.
One says that $X$ has the \emph{Moore-Myhill property}\index{Moore-Myhill property!--- for a subshift}\index{subshift!--- satisfying the
Moore-Myhill property}  
or that it satisfies the \emph{Garden of Eden theorem} if it has both the Moore and the Myhill properties.\index{Garden of Eden theorem!subshift satisfying the ---}\index{subshift!--- satisfying the Garden of Eden theorem}

From Theorem~\ref{t:fiorenzi-strongly}, we immediately deduce the following.

\begin{corollary}
\label{c:MM-sftsi}
Let $G$ be a countable amenable group and $A$ a finite set.
Then every strongly irreducible subshift of finite type $X \subset A^G$ has the Moore-Myhill property. 
\end{corollary}

\begin{example}
Let $G = \Z^d$ and $A = \{0,1\}$.
Consider the hard-ball model $X \subset A^G$ described in Example~\ref{ex:hard-ball}.\index{hard-ball model}\index{subshift!hard-ball model ---}
As $\Z^d$ is amenable and $X$ is both strongly irreducible and of finite type,
we deduce from Corollary~\ref{c:MM-sftsi} that $X$ has the Moore-Myhill property.
In particular ($d = 1$), the golden mean subshift has the Moore-Myhill property.\index{golden mean subshift}\index{subshift!golden mean ---}  
\end{example}

\begin{example}[Fiorenzi]
\label{exemple-fiorenzi-5}
Let $A = \{0,1\}$ and let $X \subset A^\Z$ be the even subshift 
(cf.\ Example~\ref{ex:even-subshift}).\index{even subshift}\index{subshift!even ---}
Consider the cellular automaton  $\sigma \colon A^\Z \to A^\Z$ with memory set
$S = \{0,1,2,3,4\}$ and local defining map $\mu \colon A^S \to A$ given by
\[
\mu(y) =
\begin{cases}
1 & \mbox{ if $y(0)y(1)y(2) \in \{000, 111\}$ or $y(0)y(1)y(2)y(3)y(4) = 00100$}\\
0 & \mbox{ otherwise.}
\end{cases}
\]
Then one has $\sigma(X) \subset X$, and the cellular automaton $\tau \coloneqq \sigma\vert_X  \colon X \to X$ is not
pre-injective. Indeed, the configurations $x_1, x_2 \in X$ defined by
\[
x_1 = \cdots 0 \cdots 00(100)100 \cdots 0 \cdots
\]
and
\[
x_2 =  \cdots 0 \cdots 00(011)100 \cdots 0 \cdots
\]
satisfy
\[
\tau(x_1) =  \cdots 1 \cdots 11(100)10011 \cdots 1 \cdots = \tau(x_2).
\]
Observe, alternatively, that the patterns $p,q$ with support $\Omega \coloneqq \{0,1, \ldots, 12\}$ defined by
\[
p(n) = \begin{cases} 1  & \mbox{ if } n = 6,9\\
0 & \mbox{ otherwise} \end{cases} \mbox{ \ \ and \ \ } q(n) = \begin{cases} 1  & \mbox{ if } n = 7,8,9\\
0 & \mbox{ otherwise,} \end{cases}
\]
for all $n \in \Omega$, are ME. 
\par
From a case-by-case analysis, one can show that $\tau$ is surjective.
It follows that $X$ does not have the Moore property.
\par
We refer to \cite[Section 3]{fiorenzi-sofic} and \cite[Counterexample~2.18]{CFS-goe} for more details.
\end{example}

  As the even subshift is strongly irreducible and $\Z$ is amenable, the previous example shows that
  the hypothesis that $X$ is of finite type cannot be removed from Corollary~\ref{c:MM-sftsi}. 
However, we have the following (cf.~\cite{csc-myhill-monatsh}).

\begin{theorem}
\label{t:strong-irr-myhill}
Let $G$ be a countable amenable group and  $A$  a finite set. 
Then every strongly irreducible subshift $X \subset A^G$ has the Myhill property.
\end{theorem}

\begin{example}
The even subshift $X \subset \{0,1\}^\Z$ has the Myhill property since it is strongly irreducible 
and $\Z$ is amenable.\index{even subshift}\index{subshift!even ---}
\end{example}

\begin{example}
\label{exemple:GOE-contrex-preinj-notsurj-1}
Let $A = \{0,1\}$. 
Let $x_0,x_1 \in A^\Z$ denote the two constant configurations 
respectively defined by $x_0(n) = 0$ and
$x_1(n) = 1$ for all $n \in \Z$. Note that
$X = \{x_0,x_1\}$ is a subshift of finite type.
The map $\tau \colon X \to X$ given by $\tau(x_0) = \tau(x_1) = x_0$ is a cellular automaton
which is pre-injective but not surjective.
It follows that $X$ does not have the Myhill property.
This very simple example shows that the hypothesis that $X$ is strongly irreducible cannot be removed neither from Corollary~\ref{c:MM-sftsi}
nor from Theorem~\ref{t:strong-irr-myhill}.
Note that $X$ has the Moore property since $X$ is finite, so that every surjective self-mapping of $X$ is injective and therefore pre-injective.
\end{example}

\begin{example}[Fiorenzi]
Let $A = \{0,1,2\}$ and let $X \subset A^\Z$ be the subshift of finite
type consisting of all $x \in A^\Z$ such that
\[
x(n) x(n + 1) \notin \{01,02\} \quad \text{ for all } n \in \Z.
\]
Thus a configuration $x \colon \Z \to A$ is in $X$ if and only if one of
the following conditions is satisfied:
\begin{enumerate}[{\rm (i)}]
\item $x(n) = 0$ for all $n \in \Z$;
\item $x(n) \neq 0$ for all $n \in \Z$;
\item there exists $n_0 \in \N$ such that $x(n) \in \{1,2\}$ for all 
$n \leq n_0$ and $x(n) = 0$ for all $n > n_0$.
\end{enumerate}
Consider the cellular automaton $\sigma \colon A^\Z \to A^\Z$ with memory
set $S = \{0,1\}$ and local defining map
\[
\mu(y) = \begin{cases} y(0) & \mbox{ if } y(1) \neq 0\\
0 & \mbox{ otherwise.}
\end{cases}
\]
Observe that $\sigma(x) = x$ if $x \in X$ is of type (i) or (ii)
while, if $x \in X$ is of type (iii), then $\sigma(x)$ is obtained from $x$ by replacing its rightest  
nonzero term by $0$.   
We deduce  that  $\sigma(X) \subset X$ and
that the  cellular automaton $\tau \coloneqq \sigma\vert_X \colon X \to X$ is surjective but not
pre-injective (see \cite[Counterexample~4.27]{fiorenzi-sofic}). It follows that $X$ does not have the Moore property.
\par
It turns out that $X$ does not have the Myhill property either.
Indeed, consider now the cellular automaton $\sigma' \colon A^\Z \to A^\Z$ with memory
set $S' = \{-1,0\}$ and local defining map
\[
\mu'(y) = \begin{cases} y(0) & \mbox{ if } y(-1)y(0) \notin  \{10,20\} \\
y(-1) & \mbox{ otherwise.}
\end{cases}
\]
Observe that $\sigma'(x) = x$ if $x \in X$ is of type (i) or (ii)
while, if $x \in X$ is of type (iii), then $\sigma'(x)$ is obtained from $x$ by repeating on its right its rightest  nonzero term.   
We deduce that
$\sigma'(X) \subset X$ and that 
the cellular automaton $\tau'
\coloneqq \sigma'\vert_X \colon X \to X$ is injective and hence pre-injective. 
However, $\tau'$ is  not surjective
(observe for instance that the pattern $p \in A^{\{-1,0,1\}}$ defined by
$p(-1)p(0)p(1) = 120$ is a Garden of Eden pattern for $\tau'$).
\end{example}

Let $A$ be a finite set and $X \subset A^\Z$ a subshift.
One says that a word $u \in A^\star$ of length $n$
\emph{appears}\index{word!--- appearing in a subshift}\index{subshift!word appearing in a ---} in $X$ if there is a configuration $x \in X$ and $m \in \Z$
such that $u = x(m)x(m+1) \cdots x(m+n - 1)$.
The subset $L(X) \subset A^\star$ consisting of all words that appear in $X$ is called the \emph{language}\index{language of a subshift}\index{subshift!language of a ---} of $X$.
One says that the subshift $X$ is \emph{irreducible}\index{irreducible subshift}\index{subshift!irreducible ---} if given any two words $u,v \in L(X)$ there exists a word $w \in L(X)$ such that
$u w v \in L(X)$.
Clearly every strongly irreducible subshift $X \subset A^\Z$ is irreducible.
The converse is false as shown by the following example.

\begin{example}
Let $A = \{0,1\}$ and consider the subshift $X \subset A^\Z$ consisting of the two configurations 
$x \in A^\Z$ that satisfy $x(n) \not= x(n + 1)$ for all $n \in \Z$.
It is clear that $X$ is irreducible but not strongly irreducible. Observe that $X$ is of finite type. 
\end{example}

The following result is an immediate consequence of \cite[Theorem~8.1.16]{lind-marcus}
(cf.~\cite[Corollary~2.19]{fiorenzi-sofic}).

\begin{theorem}
Let $A$ be a finite set.
Then every irreducible subshift of finite type $X \subset A^\Z$ has the Moore-Myhill property.
\end{theorem}

\section{Garden of Eden theorems for other dynamical systems}
\label{sec:goe-dyn-sys}

\subsection{Dynamical systems}
By a \emph{dynamical system},\index{dynamical system} we  mean a triple $(X,G,\alpha)$, where 
 $X$ is a compact metrizable space, $G$ is a countable group, 
   and  $\alpha$ is a continuous  action of $G$ on $X$.
The space $X$ is called the \emph{phase space}\index{phase space}\index{dynamical system!phase space of a ---} of the dynamical system.  
   If there is no risk of confusion, 
   we shall write $(X,G)$, or even sometimes simply $X$,  instead of $(X,G,\alpha)$. 

\begin{example}
Let  $G$ be a countable group  and $A$ a compact metrizable topological space (e.g.~a finite set with its discrete topology).
Equip $A^G = \{x \colon G \to A\}$
with the product topology.
The \emph{shift action}\index{shift!--- action}\index{shift}\index{action!shift ---} $\sigma$ of $G$ on $A^G$
is the action defined by $\sigma(g,x) = g x$, where
\[
(g x)(h) = x(g^{-1} h)
\]
for all $x \in A^G$ and $g,h \in G$.
Then $(A^G,G,\sigma)$ is a dynamical system. 
\end{example}

\begin{example}
If  $(X,G,\alpha)$ is a dynamical system and  $Y \subset X$ a closed $\alpha$-invariant subset,
then $(Y,G,\alpha|_Y)$,
where $\alpha|_Y$ denotes the action of $G$ on $Y$ induced by restriction of $\alpha$,
 is a dynamical system.
In particular, if $G$ is a countable group, $A$ a finite set,
and $X \subset A^G$ a subshift,
then  $(X,G,\sigma|_X)$ is a dynamical system.
\end{example}

\begin{example}
Let $f \colon X \to X$ be a homeomorphism of a compact metrizable space $X$.
The dynamical system \emph{generated}\index{dynamical system!--- generated by a homeomorphism} by $f$ is the dynamical system
$(X,\Z,\alpha_f)$, where $\alpha_f$ is the action of $\Z$ on $X$ given by
$\alpha_f(n,f) \coloneqq f^n(x)$ for all $n \in \Z$ and $x \in X$.
We shall also  write $(X,f)$ to denote the dynamical system generated by $f$. 
\end{example}

\begin{remark}
If we fix the countable group $G$,
the dynamical systems $(X,G)$ are the objects of a concrete category $\DD_G$
in which the morphisms from an object $X \in \DD_G$ to another object $Y \in \DD_G$ consist of all 
equivariant continuous maps $\tau \colon X \to Y$. 
It follows from the Curtis-Hedlund-Lyndon theorem (cf.~Theorem~\ref{t:curtis-hedlund-lyndon}) that the category $\CC_G$ described in 
Remark~\ref{rem:category-subshifts} is a full subcategory of the category  $\DD_G$.
\end{remark}

Let $(X,G)$ and $(\widetilde{X},G)$ be two dynamical systems.
\par
One says that the dynamical systems $(X,G)$ and 
$(\widetilde{X},G)$ are \emph{topologically conjugate}\index{topologically!--- conjugate dynamical systems}
if they are isomorphic objects in the category $\DD_G$, i.e., 
if there exists an equivariant homeomorphism 
$h \colon \widetilde{X} \to X$.
\par
One says that $(X,G)$ is a \emph{factor}\index{factor!--- of a dynamical system}\index{dynamical system!factor of a ---} of 
$(\widetilde{X},G)$ if there exists an equivariant continuous surjective map
$\theta \colon \widetilde{X} \to X$. 
Such a map $\theta$ is then called a \emph{factor map}.\index{factor!--- map}
A factor map $\theta \colon \widetilde{X} \to X$ is said to be \emph{finite-to-one}\index{factor!finite-to-one --- map}\index{finite-to-one factor map} 
if the pre-image set 
$\theta^{-1}(x)$ is finite for each $x \in X$.
A finite-to-one factor map is said to be \emph{uniformly bounded-to-one}\index{factor!uniformly bounded-to-one --- map}\index{uniformly bounded-to-one factor map}  
if there is an integer $K \geq 1$ such that
$|\theta^{-1}(x)| \leq K$ for all  $x \in X$.

\subsection{Expansiveness}
One says that a dynamical system  $(X,G)$ is \emph{expansive}\index{expansive dynamical system}\index{dynamical system!expansive ---} if there exists a neighborhood 
$W \subset X \times X$ of the diagonal
\[
\Delta_X \coloneqq \{(x,x) : x \in X\} \subset X \times X
\] 
such that,
for every pair of distinct points $x,y \in X$,
there exists an element  $g = g(x,y)  \in G$ such that
$(g x, g y) \notin W$.
Such a set $W$ is then called an
\emph{expansiveness neighborhood}\index{expansiveness neighborhood} of the diagonal.
\par 
If $d$ is a metric on $X$ compatible with the topology,
the fact that $(X,G)$ is expansive is equivalent to the existence of a constant  
$\delta  > 0$ such that, for every pair of distinct points $x,y \in X$, there exists an element
$g = g(x,y)  \in G$ such that
$d(g x,g y) \geq  \delta$.

\begin{example}
Let $G$ be a countable group and $A$ a finite set. Then the $G$-shift on $A^G$ is expansive.
Indeed, it is clear that
\[
W \coloneqq \{(x,y) \in A^G \times A^G : x(1_G) = y(1_G)\}
\]
is an expansiveness neighborhood of $\Delta_{A^G}$.
\end{example}

\begin{example}
If $(X,G)$ is an expansive dynamical system and $Y \subset X$ is a closed invariant subset, then $(Y,G)$ is expansive.
Indeed, if $W$ is an expansiveness neighborhood of $\Delta_X$, then 
$W \cap (Y \times Y)$ is an expansiveness neighborhood of $\Delta_Y$.
In particular, if $G$ is a countable group, $A$ a finite set,
and $X \subset A^G$ a subshift,
then the dynamical system $(X,G,\sigma|_X)$ is expansive.
\end{example}

\subsection{Homoclinicity}
\label{ss:homoclinicity}
Let $(X,G)$ be a dynamical system.
Two points $x,y \in X$ are called \emph{homoclinic}\index{homoclinic!--- points} with respect to the action of $G$ on $X$, ore more briefly,
$G$-\emph{homoclinic}, if for any neighborhood $W \subset X \times X$ of the diagonal $\Delta_X$, there is a finite
set $F = F(W,x,y) \subset G$ such that $(gx,gy) \in W$ for all $g \in G \setminus F$.
\par
If $d$ is a metric on $X$ that is compatible with the topology,
two points $x, y \in X$ are homoclinic  if and only if
\[
\lim_{g \to \infty} d(g x, g y) = 0,
\]
where $\infty$ is the point at infinity in the one-point compactification of the discrete group $G$.
This means that,   for every $\varepsilon > 0$, there is a finite subset $F = F(\varepsilon,d,x,y) \subset G$ such that 
$d(g x,g y) < \varepsilon $ for all $g \in G \setminus F$.
\par
Homoclinicity clearly defines an equivalence relation on $X$ (transitivity follows from the triangle inequality). 
The equivalence classes of this relation are called the $G$-\emph{homoclinicity classes}\index{homoclinicity class} 
of $X$.

\begin{definition}
\label{def:pre-inject-ds}
Let $(X,G)$ be a dynamical system and $Y$ a set.
One says that a map $\tau \colon X \to Y$ is \emph{pre-injective}\index{pre-injective!--- map} if its restriction to each $G$-homoclinicity class is injective.
\end{definition}

\begin{example}
\label{ex:homoclinic-almost-equal}
Let $G$ be a countable group and $A$  a finite set.
Two configurations $x, y \in A^G$ are homoclinic with respect to the shift action of $G$ on $A^G$ if and only if they are almost equal
(see e.g.~\cite[Proposition~2.5]{csc-ijm-expansive}).
Indeed, first observe that the sets
\[
W_\Omega \coloneqq \{(x,y) \in A^G \times A^G : x\vert_\Omega = y\vert_\Omega\},
\]
where $\Omega$ runs over all finite subsets of $G$, form a neighborhood base of the diagonal 
$\Delta_{A^G} \subset A^G \times A^G$ (this immediately follows from the definition of the product topology).
Now, if $x, y \in A^G$ are almost equal, then the set $D \subset G$ consisting of all $g \in G$ such that $x(g) \not= y(g)$ is finite, so that
$\Omega D^{-1}$ is also finite for every finite subset $\Omega \subset G$.
As $(g x, g y) \in W_\Omega$ for every $g \in G \setminus \Omega D^{-1}$, this implies that $x$ and $y$ are homoclinic.
Conversely, suppose that $x, y \in A^G$ are homoclinic.
Then there exists a finite subset $F \subset G$ such that $(g x, g y) \in W_{\{1_G\}}$ for all 
$g \in G \setminus F$.
This implies that $x(g) = y(g)$ for all $g \in G \setminus F^{-1}$, so that $x$ and $y$ are almost equal.
\end{example}

\begin{example}
\label{ex:homoclinic-sub}
Let $(X,G,\alpha)$ be a dynamical system and $Y \subset X$  a closed invariant subset.
Denote as above by $\alpha|_Y$ the restriction of $\alpha$ to $Y$.
Then two points $x,y \in Y$ are homoclinic with respect to $\alpha|_Y$ if and only if they are homoclinic with respect to $\alpha$.
In particular, if $G$ is a countable  group, $A$ a finite set, and $X \subset A^G$ a subshift,
then two configurations $x, y \in X$ are homoclinic with respect to $\sigma|_X$ if and only if they are almost equal.
It follows that the definition of pre-injectivity for cellular automata
between subshifts given in Definition~\ref{def:pre-inj-subshifts} 
agrees with  the one given in Definition~\ref{def:pre-inject-ds} above.  
\end{example}

\subsection{The Moore and the Myhill properties for dynamical systems}
Let $(X,G)$ be a dynamical system.
\par
An \emph{endomorphism}\index{endomorphism of a dynamical system}\index{dynamical system!endomorphism of a ---} of $(X,G)$ is a continuous equivariant map $\tau \colon X \to X$.
\par
One says that the dynamical system  $(X,G)$ has the \emph{Moore property}\index{Moore property!--- for a dynamical system}\index{dynamical system!Moore property for a ---} if every surjective endomorphism of $(X,G)$  is pre-injective
and that it has the \emph{Myhill property}\index{Myhill property!--- for a dynamical system}\index{dynamical system!Myhill property for a ---} if every pre-injective 
endomorphism of $(X,G)$  is surjective.
One says that $(X,G)$ has the \emph{Moore-Myhill property}\index{Moore-Myhill property!--- for a dynamical system}\index{dynamical system!Moore-Myhill property for a ---} 
or that it satisfies the \emph{Garden of Eden theorem}\index{Garden of Eden theorem!dynamical system satisfying the ---}\index{dynamical system!--- satisfying the Garden of Eden theorem} if it has both the Moore and the Myhill properties.
\par
Observe that all these properties are invariants of topological conjugacy in the sense that if the dynamical systems $(X,G)$ and $(Y,G)$ are topologically conjugate then
$(X,G)$ has the Moore (resp.~the Myhill, resp.~the Moore-Myhill) property
if and only if $(Y,G)$ has the Moore (resp.~the Myhill, resp.~the Moore-Myhill) property.

  \begin{remark}
In the particular case when $(X,G)$ is a subshift, it immediately follows from  
Theorem~\ref{t:curtis-hedlund-lyndon}  and Example~\ref{ex:homoclinic-sub} that these definitions are equivalent to the ones given in Subsection~\ref{subsec:moore-myhill-subshifts}.\end{remark}

\subsection{Anosov diffeomorphisms}
Let $f \colon M \to M$ be a diffeomorphism of a compact smooth manifold $M$.
One says that $f$ is an \emph{Anosov diffeomorphism}\index{Anosov diffeomorphism}
(see e.g.~\cite{smale-bams}, \cite{brin-stuck},   
 \cite{shub-global-stability})
 if the tangent bundle $TM$ of $M$ 
continuously splits as a direct sum $TM = E_s \oplus E_u$ of two $df$-invariant subbundles $E_s$ and $E_u$ such that, with respect to some (or equivalently any) Riemannian metric on $M$,
 the differential $df$ is exponentially  contracting on $E_s$ and exponentially expanding on $E_u$, 
 i.e.,
 there are constants  $C > 0$ and $0 < \lambda < 1$ such that
 \begin{enumerate}[(i)]
 \item
$\Vert df^n(v) \Vert \leq C\lambda^n \Vert v \Vert$,
\item
$\Vert df^{-n}(w) \Vert \leq C \lambda^n \Vert w \Vert$
\end{enumerate}
for all $x \in M$, $v \in E_s(x)$,  $w \in E_u(x)$, and $n \geq 0$.

\begin{example}[Arnold's cat]\index{Arnold's cat}\index{dynamical system!Arnold's cat ---}
Consider the matrix
\[
A = 
\begin{pmatrix}
0 & 1 \\
1 & 1 
\end{pmatrix}
\]
 and the diffeomorphism  $f$ 
of the $2$-torus $\T^2 = \R^2/\Z^2 = \R/\Z \times \R/\Z$ 
given by
\begin{equation*}
f(x) \coloneqq A x = 
 \begin{pmatrix} x_2 \\ x_1 + x_2 \end{pmatrix}
 \quad
 \text{for all }
 x = \begin{pmatrix}x_1 \\ x_2 \end{pmatrix}
\in \T^2. 
\end{equation*}
The dynamical system $(\T^2,f)$ is known as \emph{Arnold's cat}.\index{dynamical system!Arnold's cat ---}
The diffeomorphism $f$ is Anosov.
Indeed, the eigenvalues of $A$ are
$\lambda_1 = - \dfrac{1}{\varphi}$, and 
$\lambda_2 = \varphi$, where $\varphi \coloneqq \dfrac{1 + \sqrt{5}}{2}$ is the golden mean.
As $-1 < \lambda_1 < 0$ and $1 < \lambda_2$, it follows that
$df = A$ is exponentially contracting in the direction of the eigenline associated with  $\lambda_1$ and uniformly expanding in the direction of the eigenline associated with  $\lambda_2$.
\end{example}

\begin{example}[Hyperbolic toral automorphism]\index{hyperbolic toral automorphism}
More generally, if $A \in \GL_n(\Z)$ has no eigenvalue on the unit circle,
then the diffeomorphism $f$ of the $n$-torus $\T^n = \R^n/\Z^n$, defined by $f(x) = A x$ for all $x \in \T^n$, is Anosov.
Such a  diffeomorphism is called a \emph{hyperbolic toral automorphism}.  
\end{example}

In~\cite[Theorem~1.1]{csc-goe-anosov},
we obtained the following result.

\begin{theorem}
\label{t:goe-anosov-tori}
Let $f$ be an Anosov diffeomorphism of the $n$-dimensional torus $\T^n$.
Then the dynamical system $(\T^n,f)$ has the Moore-Myhill property.
\end{theorem}

The proof given in~\cite{csc-goe-anosov} uses two  classical results.
The first one is the Franks-Manning theorem~\cite{franks}, \cite{manning}\index{Franks-Manning theorem}\index{theorem!Franks-Manning ---} 
which states that
$(\T^n,f)$ is topologically conjugate to a hyperbolic toral automorphism. 
The second one is a result of Walters~\cite{walters} which says that every endomorphism of a hyperbolic toral automorphism is affine.
\par
We do not know if  the dynamical system $(M,f)$ has the Moore-Myhill property whenever $f$ is an Anosov diffeomorphism of a compact smooth manifold $M$.
However, we have obtained in \cite[Theorem~1.1]{csc-ijm-expansive} the following result.

\begin{theorem}
\label{t:myhill-hyp}
Let $X$ be a compact metrizable space equipped  with a continuous action of a countable amenable group $G$. 
Suppose that the dynamical system $(X,G)$ is expansive 
and that there exist a finite set $A$, a strongly irreducible subshift 
$\widetilde{X} \subset A^G$, and a uniformly bounded-to-one factor map 
$\theta \colon \widetilde{X} \to X$.
Then the dynamical system $(X,G)$ has the Myhill property.
\end{theorem}

A homeomorphism $f$ of a topological space $X$ is called \emph{topologically mixing}\index{topologically!--- mixing dynamical system}\index{dynamical system!topologically mixing ---} if, given any two nonempty open subsets 
$U,V \subset X$, there exists an integer $N \geq 0$ such that
$f^n(U) \cap V \not= \varnothing$ for all $n \in \Z$ that satisfy  $|n| \geq N$.
By the classical work of Bowen\index{Bowen, Robert E.\ (Rufus)} (cf. \cite[Theorem~28 and Proposition~30]{bowen-markov-1970} and  \cite[Proposition~10]{bowen-markov-minimal_AJM-1970}),  dynamical systems generated by topologically mixing Anosov diffeomorphisms satisfy the hypotheses of
Theorem~\ref{t:myhill-hyp}. 
As a consequence (cf.~\cite[Corollary~4.4]{csc-goe-anosov}), we get the following partial extension of 
Theorem~\ref{t:goe-anosov-tori}.

\begin{corollary}
\label{cor:goe-anosov}
Let  $f$ be a topologically mixing Anosov diffeomorphism of a 
compact smooth manifold $M$.
Then the dynamical system $(M,f)$ has the Myhill property.
\end{corollary}

\begin{remark}
All known examples of Anosov diffeomorphisms are topologically mixing.
Also, all compact smooth manifolds that are known to admit Anosov diffeomorphisms are infra-nilmanifolds. 
We recall that a \emph{nilmanifold}\index{nilmanifold} is a manifold  of the form $N/\Gamma$, where $N$ is a simply-connected nilpotent Lie group and $\Gamma$ is a discrete cocompact subgroup of $N$ and that a \emph{infra-nilmanifold}\index{infra-nilmanifold}  is a manifold that is finitely covered by some nilmanifold.
\end{remark}

\subsection{Weak specification}
\label{sec:WS}
Recently, after our preprint \cite{csc-goe-principal} had circulated, Hanfeng Li\index{Li, Hanfeng} posted his paper \cite{li-2017} containing an impressive Garden of Eden type theorem generalizing several results mentioned above (see Theorems \ref{t:Li-1}
and \ref{t:Li-2}, and Corollaries \ref{c:Li-1} and \ref{c:Li-2}).

The key notion in Li's paper is that of \emph{specification}, a strong orbit tracing property which was introduced
by Rufus Bowen\index{Bowen, Robert E.\ (Rufus)} for $\Z$-actions in relation to his studies on Axiom A diffeomorphisms in \cite{bowen-periodic-1971} (see
also \cite[Definition 21.1]{denker-grillenberger-sigmund}) and was subsequently extended to $\Z^d$-actions by Ruelle\index{Ruelle, David} in 
\cite{ruelle-statistical-1973}. Several versions and generalizations of specification have appeared in the literature 
(see, in particular, \cite[Definition 5.1]{lind-schmidt} and \cite[Definition 6.1]{chung-li-2015}).
Here is the one we need (cf.~\cite[Definition 6.1]{chung-li-2015}).

\begin{definition}
\label{def:weak-specification}
A dynamical system $(X,G)$ has the \emph{weak specification property}\index{weak specification property}\index{dynamical system!weak specification property for a ---} 
if for any $\varepsilon > 0$ there exists a nonempty
symmetric finite subset $\Delta \subset G$ satisfying the following property: if $(\Omega_i)_{i \in I}$ is any finite family
of finite subsets of $G$ such that $\Delta \Omega_i \cap \Omega_j = \varnothing$ for all distinct $i,j \in I$, and $(x_i)_{i \in I}$ 
is any family of points in $X$, then there exists $x \in X$ such that 
\[
d(sx,sx_i) \leq \varepsilon \ \ \mbox{ for all } i \in I \mbox{ and } s \in \Omega_i,
\]
where $d$ is any metric compatible with the topology on $X$.
\end{definition}

It is straightforward (cf.~\cite[Proposition A.1]{li-2017}) to check that if $G$ is a countable group, $A$ is a finite
alphabet set, and $X \subset A^G$ is a subshift, then the shift dynamical system $(X,G)$ has the weak specification property 
if and only if it is strongly irreducible (cf.~Section \ref{sec:goe-subshifts}). 
Also, it is easy to see that the weak specification property passes to factors.

Li \cite[Theorem 1.1]{li-2017} proved the following:

\begin{theorem}
\label{t:Li-1}
Let $(X,G)$ be a dynamical system. Suppose that the group $G$ is amenable and that $(X,G)$ is expansive and has the
weak specification property. Then $(X,G)$ has the Myhill property.
\end{theorem}

Note that Theorem \ref{t:Li-1} covers Theorem \ref{t:myhill-hyp}, by virtue of the remarks following 
Definition \ref{def:weak-specification}.  

Recall (cf.~Example \ref{exemple-fiorenzi-5}) that if $X \subset \{0,1\}^\Z$ denotes the even subshift 
(cf.~Example~\ref{ex:even-subshift}),  then $(X, \Z)$ is expansive, has the weak specification property
(since $X$ is strongly irreducible), but does not have the Moore property. 
This shows that from the hypotheses of Theorem \ref{t:Li-1} one cannot deduce the Moore property, in general.

\subsection{Algebraic dynamical systems}
\label{sec:ADS}
An \emph{algebraic dynamical system}\index{algebraic dynamical system}\index{dynamical system!algebraic ---} is a dynamical system  of the form $(X,G)$, where $X$ is a compact metrizable abelian topological group
and $G$ is a countable group acting on $X$ by continuous group morphisms.
Note that if $(X,G)$ is an algebraic dynamical system, then, for each $g \in G$, the map 
$x \mapsto g x$ is a continuous group automorphism of $G$ with inverse $x \mapsto g^{-1} x$.

\begin{example}
Let $G$ be a countable group and $A$ a compact metrizable topological group
(for example a finite discrete abelian group, or  the $n$-dimensional torus $\T^n$, or  the infinite-dimensional torus $\T^\N$,  or  the group $\Z_p$  of $p$-adic integers for some prime $p$).
Then the $G$-shift $(A^G,G)$ is an algebraic dynamical system. 
\end{example}  

\begin{example}
Let $X$ be a compact metrizable abelian group and $f \colon X \to X$ a continuous group automorphism (for example $X = \T^n$ and $f \in \GL_n(\Z)$).
Then the dynamical system $(X,f)$ generated by $f$ is an algebraic dynamical system.
\end{example}

Let $(X,G)$ be an algebraic dynamical system.
\par
If $d$ is a metric on $X$ that is compatible with the topology then a point $x \in X$ is homoclinic to $0_X$ if and only if one has
\[
\lim_{g \to \infty} d(g x, 0_X) = 0.
\]
The set $\Delta(X,G)$ consisting of all points of $X$ that are homoclinic to $0_X$ is an 
$G$-invariant additive subgroup of $X$, called
the \emph{homoclinic group}\index{homoclinic!--- group}\index{group!homoclinic ---} of $(X,G)$
(cf.~\cite{lind-schmidt}).
Two points $x,y \in X$ are homoclinic if and only if $x - y \in \Delta(X,G)$.
It follows that the set of homoclinicity classes of $(X,G)$ can be identified with the quotient group
$X/\Delta(X,G)$.
 \par
Consider now the Pontryagin dual $\widehat{X}$ of $X$.
We recall that if $L$ is a locally compact abelian group,\index{locally compact abelian group}\index{group!locally compact abelian ---} the elements of its Pontryagin dual $\widehat{L}$\index{Pontryagin dual}\index{locally compact abelian group!Pontryagin dual of a ---}\index{group!Pontryagin dual of a locally compact abelian ---} are 
the \emph{characters}\index{character of a locally compact abelian group}\index{locally compact abelian group!character of a ---}\index{group!character of a locally compact abelian ---} of $L$, i.e., the continuous group morphisms
$\chi \colon L \to \T$, where $\T \coloneqq \R/\Z$,
and that the topology on $\widehat{L}$ is the topology of uniform convergence on compact subsets
(see e.g.~\cite{morris}).
As the abelian group $X$ is  compact and  metrizable, 
$\widehat{X}$ is a discrete countable abelian group.
There is also a natural dual action of $G$ on $\widehat{X}$ defined by
\[
g \chi(x) \coloneqq \chi(g^{-1} x)
\]
for all $g \in G$, $\chi \in \widehat{X}$, and $x \in X$.
Note that $\chi \mapsto g \chi$ is a group automorphism of $\widehat{X}$ for each $g \in G$.
\par
We recall that the \emph{integral group ring}\index{integral group ring}\index{group!integral --- ring} $\Z[G]$ of $G$  consists of all formal series
\[
r = \sum_{g \in G} r_g g,
\]
where $r_g \in \Z$ for all $g \in G$ and $r_g = 0$ for all but finitely many $g \in G$, and  the  operations on $\Z[G]$ are defined by the formulas
\begin{align}
\label{e:sum+convolution}
r + s &= \sum_{g \in G} (r_g + s_g) g,  \\
r s &= \sum_{g_1,g_2 \in G} r_{g_1} s_{g_2}  g_1 g_2 
\end{align}
for all
\[
r = \sum_{g \in G} r_g g, \quad s = \sum_{g \in G} s_g g \in \Z[G].
\]
By linearity, the action of $G$ on $\widehat{X}$ extends to a left $\Z[G]$-module structure on 
$\widehat{X}$.
\par
Conversely, if $M$ is a countable left $\Z[G]$-module and we equip $M$ with its discrete topology, then its Pontryagin dual $\widehat{M}$ is a compact metrizable abelian group.
The left $\Z[G]$-module structure on $M$ induces by restriction an action of $G$ on $M$, and, by dualizing, we get an action of $G$ on $\widehat{M}$ by continuous group morphisms, 
so that $(\widehat{M},G)$ is an algebraic dynamical system.
\par
Using the fact that every locally compact abelian group is isomorphic to its bidual,
one shows that Pontryagin duality yields a one-to-one correspondence between algebraic dynamical systems with acting group $G$ and  countable left $\Z[G]$-modules
(see~\cite{schmidt-book}, \cite{lind-schmidt-handbook}, \cite{lind-schmidt-survey-heisenberg}).
\par

Recall that from the hypotheses of Theorem \ref{t:Li-1} one cannot deduce the Moore property, in general.
However, Li \cite[Theorem 1.2]{li-2017} proved that when restricting to the class of algebraic dynamical 
systems (cf.~Section \ref{sec:PADS}) with amenable acting group, the Moore property follows from 
expansiveness and weak specification:

\begin{theorem}
\label{t:Li-2}
Let $(X,G)$ be an algebraic dynamical system. Suppose that the group $G$ is amenable and that $(X,G)$ is expansive and has the
weak specification property. Then $(X,G)$ has the Moore property.
\end{theorem}

As an immediate consequence of Theorems \ref{t:Li-1} and \ref{t:Li-2}, one deduces the following (cf.~\cite[Theorem 1.3]{li-2017}):

\begin{corollary}[Garden of Eden theorem for expansive algebraic dynamical systems with the weak specification property]
\label{c:Li-1}\index{Garden of Eden theorem!--- for expansive algebraic dynamical systems with the weak specification property}\index{theorem!Garden of Eden --- for expansive algebraic dynamical systems with the weak specification property}
Let $(X,G)$ be an algebraic dynamical system. Suppose that the group $G$ is amenable and that $(X,G)$ is expansive and has the
weak specification property. Then $(X,G)$ has the Moore-Myhill property.
\end{corollary}

\subsection{Principal algebraic dynamical systems}
\label{sec:PADS}

Let $f \in \Z[G]$ and consider the cyclic left $\Z[G]$-module $M_f \coloneqq \Z[G]/ \Z[G]f$ obtained by quotienting the ring $\Z[G]$ by the principal left ideal generated by $f$.
The algebraic dynamical system associated by Pontryagin duality with $M_f$ is denoted by $(X_f,G)$ and is called the \emph{principal algebraic dynamical system}\index{principal algebraic dynamical system}\index{algebraic dynamical system!principal ---} associated with $f$.
\par 
There is a beautiful characterization of expansivity for principal algebraic dynamical systems due to Deninger and Schmidt~\cite[Theorem~3.2]{deninger-schmidt} (see also \cite[Theorem~5.1]{lind-schmidt-survey-heisenberg}).
Let $G$ be a countable group and $f \in \Z[G]$. Then $(X_f,G)$ is expansive if and only if $f$ is invertible in $\ell^1(G)$.
(Here $\ell^1(G)$ denotes the Banach algebra consisting of all formal sums $r = \sum_{g \in G}r_g g$ such that 
$r_g \in \R$ for all $g \in G$ and $\Vert r \Vert_1 \coloneqq \sum_{g \in G} |r_g| < \infty$, 
equipped with its obvious real vector space structure and the convolution product as in \eqref{e:sum+convolution}.)

It turns out (cf.\ \cite[Lemma 2.1]{li-2017}, see also \cite[Theorem 1.2]{ren-2015}), that every expansive principal algebraic action
has the weak specification property. From Corollary \ref{c:Li-1} one immediately deduces (cf.~\cite[Theorem 1.3]{li-2017}):

\begin{corollary}[Garden of Eden theorem for principal expansive algebraic dynamical systems]
\label{c:Li-2}\index{Garden of Eden theorem!--- for principal expansive algebraic dynamical systems}\index{theorem!Garden of Eden --- for principal expansive algebraic dynamical systems}
Let $(X,G)$ be a principal algebraic dynamical system. Suppose that the group $G$ is amenable and that $(X,G)$ is expansive.  
Then $(X,G)$ has the Moore-Myhill property.
\end{corollary}

In \cite[Theorem 1.1]{csc-goe-principal} we had proved the same result under the stronger assumptions that $G$ is abelian and the
phase space $X$ is connected.
\par
In the case  $G = \Z^d$, the group ring $\Z[G]$ can be identified with the ring $\Z[u_1, u_1^{-1}, \dots, u_d, u_d^{-1}]$ of Laurent polynomials with integral coefficients on $d$ commuting indeterminates.

\begin{example}
\label{e:principal-arnold-cat}\index{Arnold's cat}\index{dynamical system!Arnold's cat ---}
For $G = \Z$ and $f = u^2 - u - 1 \in \Z[u,u^{-1}] = \Z[G]$, one can check that the associated principal algebraic dynamical system
$(X_f,\Z)$ is topologically conjugate to Arnold's cat on $\T^2$ (see e.g.~\cite[Example~2.18.(2)]{schmidt-book}).
Thus, from Corollary~\ref{c:Li-2} we recover that Arnold's cat satisfies the Moore-Myhill property.
\end{example}

In \cite{cscl-goe-homoclinically}, in collaboration with Hanfeng Li,\index{Li, Hanfeng} we introduced a notion of weak expansivity for elements in the integral group ring $\Z[G]$, for any countable group $G$, and proved a Garden of Eden theorem for principal algebraic dynamical systems associated with weakly expansive polynomials. In order to state it, let us first introduce some preliminary material and notation.
We denote by $\CC_0(G)$ the real vector space consisting of all maps $r \colon G \to \R$ such that $\lim_{g \to \infty} r(g) = 0$
(this means that for all $\varepsilon > 0$ there exists a finite subset $F \subset G$ such that $|r(g)| < \varepsilon$ for all
$g \in G \setminus F$). 
Note that if $r \in \CC_0(G)$ and $s \in \Z[G]$ then the map $rs \colon G \to \R$ defined by 
$(rs)(g) = \sum_{g_1,g_2 \in G} r(g_1) s_{g_2}$ for all $g \in G$ (cf.\ \eqref{e:sum+convolution}) belongs to $\CC_0(G)$. 
This endows $\CC_0(G)$ with a structure of a  right $\Z[G]$-module. Moreover $G \subset \Z[G] \subset \CC_0(G)$.

\begin{definition}
\label{def:we}
An element $f \in \Z[G]$ is said to be {\it weakly expansive}\index{weakly expansive polynomial} provided:
\begin{enumerate}[{\rm (we-1)}]
\item $\forall r \in \CC_0(G)$, $fr = 0$ $\Rightarrow$ $r = 0$;
\item $\exists \omega \in \CC_0(G)$ such that $f\omega = 1_G$.
\end{enumerate}
\end{definition}

For principal algebraic dynamical systems with elementary amenable acting group there is a characterization of connectedness of the phase space. First recall that a non-zero element $f \in \Z[G]$ is called \emph{primitive}\index{primitive polynomial} if there is no integer $n \geq 2$ that divides all coefficients of $f$. Also recall (cf.\ for instance \cite{Chou}) that the class of \emph{elementary amenable groups}\index{amenable group!elementary ---}\index{group!elementary amenable ---}\index{elementary!--- amenable group} is the smallest class of groups containing all finite groups and all Abelian groups that is closed under the operations of taking subgroups, quotiens, extensions, and direct limits. In \cite[Proposition 2.4]{cscl-goe-homoclinically} we showed that is $G$ is a countable torsion-free elementary amenable group (e.g.\ $G = \Z^d$) and $f \in \Z[G]$ is non-trivial, then $X_f$ is connected if and only if $f$ is primitive.

We are now in position to state the main result of \cite{cscl-goe-homoclinically} (Theorem 1.1 therein).

\begin{theorem}[Garden of Eden theorem for algebraic actions associated with weakly expansive polynomials]
\label{t:main-result-cscl}\index{Garden of Eden theorem!--- for algebraic actions associated with weakly expansive polynomials}\index{theorem!Garden of Eden --- for algebraic actions associated with weakly expansive polynomials}
Let $G$ be a countable Abelian group and $f \in \Z[G]$.
Suppose that $f$ is weakly expansive and that $X_f$ is connected.
Then the dynamical system $(X_f,G)$ has the Moore-Myhill property.
\end{theorem}

There are two main ingredients in our proof of Theorem \ref{t:main-result-cscl}.
The first one is a rigidity result (a generalization of \cite[Corollary 1]{bhattacharya})
for algebraic dynamical systems associated with weakly expansive polynomials and with connected phase space.
We used it to prove that, under the above conditions, every endomorphism of $(X_f,G)$ is affine with linear part of the form
$x \mapsto r x$ for some $r \in \Z[G]$.
The second one, a generalization of \cite[Lemma 4.5]{lind-schmidt}),
asserts that, if $f$ is weakly expansive, then the homoclinic group $\Delta(X_f,G)$, equipped with
the induced action of $G$, is dense in $X_f$ and isomorphic, as a $\Z[G]$-module, to $\Z[G]/ \Z[G] f^*$, where
$f^* \in \Z[G]$ is defined by $(f^*)_g \coloneqq f_{g{-1}}$ for all $g \in G$.

In \cite[Corollary 3.12]{cscl-goe-homoclinically} we showed that if $f \in \Z[G]$ and the associated principal algebraic dynamical system 
$(X_f,G)$ is expansive then $f$ is weakly expansive. It follows that Theorem \ref{t:main-result-cscl} constitutes a generalization of the main result in \cite{csc-goe-principal}.

Recall that a polynomial $f \in \R[G]$ is said to be \emph{well-balanced}\index{well-balanced polynomial} (cf.\ \cite[Definition 1.2]{Bowen-Li}) if the following conditions are satisfied:
\begin{enumerate}[{\rm (wb-1)}]
\item $\sum_{g \in G} f_g = 0$,
\item $f_g \leq 0$ for all $g \in G \setminus \{1_G\}$,
\item $f_g = f_{g^{-1}}$ for all $g \in G$ (i.e., $f$ is \emph{self-adjoint}),
\item and $\supp(f) \coloneqq \{g \in G: f_g \neq 0\}$, the \emph{support} of $f$, generates $G$.\index{support!--- of a polynomial}
\end{enumerate}

If $f\in \Z[G]$ is well-balanced, the associated dynamical system $(X_f, G)$ is called a \emph{harmonic model}.\index{harmonic model}\index{dynamical system!harmonic model ---}
For $G = \Z^d$, the Laurent polynomial $f = 2d - \sum_{i=1}^d (u_i + u_i^{-1}) \in \Z[u_1, u_1^{-1}, \ldots, u_d, u_d^{-1}] = \Z[\Z^d]$ is
well-balanced and the corresponding harmonic model $(X_f,\Z^d)$  shares many interesting measure theoretic and entropic properties with other different models in mathematical physics, probability theory, and dynamical systems such as the Abelian sandpile model, spanning trees,
and the dimer models \cite{schmidt-verbitskiy, Bowen-Li}.
Since a well-balanced polynomial $f \in \Z[G]$, with $G$ infinite countable not virtually $\Z$ or $\Z^2$,
is weakly expansive (\cite[Proposition 3.14]{cscl-goe-homoclinically}), from Theorem \ref{t:main-result-cscl} we deduce 
(cf.\ \cite[Corollary 1.4]{cscl-goe-homoclinically}):

\begin{corollary}[Garden of Eden theorem for harmonic models]
\label{c:harmomic}\index{Garden of Eden theorem!--- for harmonic models}\index{theorem!Garden of Eden --- for harmonic models}
Let $G$ be an infinite countable Abelian group which is not virtually $\Z$ or $\Z^2$  (e.g.~$G = \Z^d$, with $d \geq 3$).
Suppose that $f \in \Z[G]$ is well-balanced and that $X_f$ is connected.
Then the dynamical system $(X_f,G)$ has the Moore-Myhill property.
\end{corollary}

If $G = \Z^d$, then any polynomial $f \in \R[G]$ may be regarded, by duality, as a function on $\widehat{G} = \T^d$.
We denote by $Z(f) \coloneqq \{(t_1, t_2, \ldots, t_d) \in \T^d: f(t_1, t_2, \ldots, t_d) = 0\}$ its zero-set.
Recall that an irreducible polynomial $f$ is \emph{atoral}\index{atoral polynomial} \cite[Definition 2.1]{lind-schmidt-verbitskiy-2} if there is some
$r\in \Z[G]$ such that $r \not \in \Z[G] f$ and $Z(f)\subset Z(r)$. This is equivalent to the condition
$\dim Z(f)\le d-2$, where the meaning of $\dim(\cdot)$ is explained in \cite[page 1063]{lind-schmidt-verbitskiy-2}; in particular,
one has $\dim(\varnothing) \coloneqq - \infty$.
Also remark that, if $d = 1$, an irreducible polynomial $f \in \Z[\Z] = \Z[u_1,u_1^{-1}]$ is atoral if and only if $Z(f) = \varnothing$
and this, in turn, is equivalent to $(X_f,\Z)$ being expansive (cf.\ \cite[Lemmma 2.1.(1)]{lind-schmidt}).
\par
We are now in position to state the following (cf.\ \cite[Theorem 1.5]{cscl-goe-homoclinically}):

\begin{theorem}[A Garden of Eden theorem for irreducible atoral polynomials]
\label{t:GOE-irr-finite-zero-set}\index{Garden of Eden theorem!--- for irreducible atoral polynomials}\index{theorem!Garden of Eden --- for irreducible atoral polynomials}
Let $f\in \Z[\Z^d]$ be an irreducible atoral polynomial such that $Z(f)$ is contained in the image of the intersection of $[0,1]^d$ and a finite union of hyperplanes in $\R^d$ under the natural quotient map $\R^d \to \T^d$ (e.g., when $d \geq 2$ such that $Z(f)$ is finite).
Then the dynamical system $(X_f,\Z^d)$ has the Moore-Myhill property.
\end{theorem}

\begin{examples}
Here below, we present some examples of irreducible atoral polynomials $f \in \Z[\Z^d]$, mainly from 
\cite[Section 3]{lind-schmidt-verbitskiy} and \cite[Section 4]{lind-schmidt-verbitskiy-2}. We can then apply Theorem \ref{t:GOE-irr-finite-zero-set} and deduce
that the corresponding algebraic dynamical systems $(X_f, \Z^d)$ satisfy the Garden of Eden theorem.

\begin{enumerate}[{\rm (1)}]
\item Let $d=1$ and $f(u) = u^2 - u - 1 \in \Z[u,u^{-1}] = \Z[\Z]$ (cf.\ Example \ref{e:principal-arnold-cat}).\index{Arnold's cat}\index{dynamical system!Arnold's cat ---}
Then $f$ is irreducible and, since $Z(f) = \varnothing$, atoral.
Recall that the associated principal algebraic dynamical system $(X_f,\Z)$ is conjugated to \emph{Arnold's cat}.
Thus we get yet another proof of the fact that this hyperbolic dynamical systems satisfies the Garden of Eden theorem.

\item Let $d = 2$ and $f(u_1,u_2) = 2 - u_1 - u_2 \in \Z[u_1,u_1^{-1},u_2,u_2^{-1}] = \Z[\Z^2]$. Then
$Z(f) = \{(1,1)\}$, and so $f$ is atoral.
Note that, in fact, $f$ is weakly expansive (though not well-balanced) by \cite[Example 6.2.(2)]{cscl-goe-homoclinically}.
Moreover, $f$ is also primitive, so that, by the characterization we presented above, $X_f$ is connected.
Applying Theorem \ref{t:main-result-cscl}, we obtain an alternative proof of the fact that $(X_f,\Z^2)$ has the Moore-Myhill property.

\item Let $d=2$, and consider the \emph{Laplace harmonic model} $f(u_1,u_2) = 4 - u_1-u_1^{-1}-u_2-u_2^{-1} \in \Z[u_1,u_1^{-1},u_2,u_2^{-1}] = \Z[\Z^2]$.\index{Laplace harmonic model}\index{dynamical system!Laplace harmonic model ---} 
One has $Z(f)=\{(1, 1)\}$. Thus $f$ is atoral and $(X_f, \alpha_f)$ satisfies the Garden of Eden theorem, by virtue of Theorem \ref{t:GOE-irr-finite-zero-set}. (Note that we cannot apply Theorem \ref{t:main-result-cscl}.)

\item Let $d=2$, and $f(u_1,u_2) = 1 + u_1 + u_2 \in \Z[u_1,u_1^{-1},u_2,u_2^{-1}] = \Z[\Z^2]$. Then $Z(f)=\{(\omega, \omega^2),
(\omega^2, \omega)\}$, where $\omega = \exp(2\pi i/3)$. The algebraic dynamical system $(X_f, \Z^2)$ is called the \emph{connected Ledrappier subhift}.\index{Ledrappier subshift!connected ---}\index{subshift!connected Ledrappier ---} Thus the connected Ledrappier shift satisfies the Garden of Eden theorem.
On the other hand, the \emph{(disconnected) Ledrappier shift} (cf.\ Example~\ref{ex:ledrappier-sub})\index{Ledrappier subshift}\index{subshift!Ledrappier ---}
$X \coloneqq \{x \in (\Z/2\Z)^{\Z^2}: x(m,n)+x(m+1,n)+x(m,n+1)=0\}$ (which may be regarded as an algebraic dynamical
system with phase space $\widehat{\Z[\Z^2]/I}$, where $I = 2\Z[\Z^2] + f\Z[\Z^2]$ is the ideal generated by $2$ and 
$f(u_1,u_2) = 1+u_1+u_2 \in \Z[u_1,u_1^{-1},u_2,u_2^{-1}] = \Z[\Z^2]$) does not satisfy the Garden on Eden theorem. 
Indeed, one has $\Delta(X,\Z^2) = \{0_{(\Z/2\Z)^{\Z^2}}\}$ so that every map $\tau \colon X \to X$ is pre-injective. 
This ensures the Moore property for $(X,\Z^2)$. 
However, the constant map $x \mapsto 0_{(\Z/2\Z)^{\Z^2}}$ (which is a pre-injective endomorphism of $(X,\Z^2)$) is clearly not surjective, showing that $(X,\Z^2)$ does not satisfy the Myhill property.

\item Let $d=3$ and $f(u_1,u_2,u_3) = 3 + 3u_1 - 3 u_1^3 + u_1^4 - u_2 - u_3  \in \Z[u_1,u_1^{-1},u_2,u_2^{-1}, u_3,u_3^{-1}] = \Z[\Z^3]$.
One has has $Z(f)=\{(\eta, \overline{\eta}, \overline{\eta}), (\overline{\eta},\eta,\eta)\}$, where $\eta$ is an algebraic integer.
\end{enumerate}
\end{examples}

\begin{remark} 
\label{r:new}
Let $d=1$ and $f = 2 - u - u^{-1} \in \Z[u,u^{-1}] = \Z[\Z]$.\index{Laplace harmonic model}\index{dynamical system!Laplace harmonic model ---}
Then the associated dynamical system $X_f = \{x \in \T^\Z: x(n-1)+x(n+1) = 2x(n) \mbox{ for all } n \in \Z\}$ is the one-dimensional
\emph{Laplace harmonic model}. It is easy to see that $\Delta(X_f,\alpha_f) = \{0_{\T^\Z}\}$.
Then $(X_f,\alpha_f)$ satisfies the Moore property but not the Myhill property
(the constant map $x \mapsto 0_{\T^\Z}$ (which is a pre-injective endomorphism of $(X_f,\alpha_f)$) is clearly not surjective).
\end{remark}

We then have (cf.\ \cite[Corollary 1.6]{cscl-goe-homoclinically}):

\begin{corollary}[Garden of Eden theorem for Laplace harmonic models]
\label{c:laplace}\index{Garden of Eden theorem!--- for Laplace harmonic models}\index{theorem!Garden of Eden --- for Laplace harmonic models}
The Laplace harmonic model\index{Laplace harmonic model}\index{dynamical system!Laplace harmonic model ---} (i.e. the principal algebraic dynamical 
system $(X_f,\alpha_f)$ associated with the polynomial 
$f = 2d - \sum_{i=1}^d (u_i + u_i^{-1}) \in \Z[u_1, u_1^{-1}, \ldots, u_d, u_d^{-1}] = \Z[\Z^d])$
satisfies the Moore-Myhill property if and only if $d \geq 2$.
\end{corollary}

\section{Some Additional topics}
\label{sec:additional}

\subsection{Infinite alphabets and uncountable groups}
The notion of a subshift and that of a cellular automaton between subshifts can be extended to the case where the alphabet sets are infinite and the group is not countable.
\par
More specifically, let $G$ be a (possibly uncountable) group and $A$ a (possibly infinite) set.
The \emph{prodiscrete topology}\index{prodiscrete!--- topology} on $A^G$ is the product topology obtained by taking the discrete topology on each factor $A$ of $A^G  = \prod_{g \in G} A$.
The prodiscrete topology on $A^G$ is not metrizable as soon as $G$ is uncountable and $A$ contains more than one element.
However, this topology is induced by the \emph{prodiscrete uniform structure}\index{prodiscrete!--- uniform structure} on $A^G$, that is, the product uniform structure on $A^G$ obtained by taking the discrete uniform structure on each factor $A$ of $A^G$ (see \cite[Appendix~B]{book} for more details).
A subset $X \subset A^G$ is called a \emph{subshift}\index{subshift} if $X$ is invariant under the shift action and closed for the prodiscrete topology.
\par  
Let  $G$ be a  group and let $A, B$ be   sets.
Suppose that $X \subset A^G$ and $Y \subset B^G$ are two subshifts.
One defines cellular automata between $X$ and $Y$ exactly as in
Definition~\ref{def:ca}.
Every cellular automaton $\tau \colon X \to Y$ is continuous with respect to the topologies on $X$ and $Y$ induced by the prodiscrete topologies on $A^G$ and $B^G$.
The converse is false in general \cite[Section~4]{csc-curtis-hedlund}, \cite[Example~1.8.2]{book}.
However, the Curtis-Hedlund-Lyndon theorem (cf.~Theorem~\ref{t:curtis-hedlund-lyndon}) admits the following generalization \cite[Theorem~1.1]{csc-curtis-hedlund}, \cite[Theorem~1.8.1]{book}: a map $\tau \colon X \to Y$ is a cellular automaton if and only if it is equivariant 
(with respect to the $G$-shift actions) and uniformly continuous
(for  the uniform structures on $X$ and $Y$ induced by the prodiscrete uniform structures on $A^G$ and  $B^G$).
\par 
One can extend the notion of amenability defined only for countable groups in Section~\ref{ss:amenable-groups} 
by declaring  that a general group  $G$ is \emph{amenable}\index{amenable group}\index{group!amenable ---} if all of its finitely generated subgroups are amenable. This extension makes sense since every finitely generated group is countable and every subgroup of a countable amenable group is itself amenable.
The Garden of Eden theorem (cf.~Theorem~\ref{t:GOE-amen})
remains valid in this more general setting:
if $G$ is a (possibly uncountable) amenable group, $A$ a finite set, and $\tau \colon A^G \to A^G$ a cellular automaton,
then $\tau \colon A^G \to A^G$ is surjective if and only if it is pre-injective.
The proof can be reduced to the case when the group $G$ is finitely generated (and hence countable) by using the operations of \emph{restriction} and \emph{induction} for cellular automata
(see~\cite{induction} and \cite[Section~1.7]{book}).
One can also directly follow the proof given above for Theorem~\ref{t:GOE-amen}
by replacing the F\o lner sequence by a F\o lner net (see \cite[Theorem~5.8.1]{book}).  

\subsection{Linear cellular automata}
Let $G$ be a group and let $K$ be a field.
Let $A$ be a finite-dimensional vector space over $K$ and set $d = \dim_K(A)$.
Observe that $A$ is infinite as soon as the field $K$ is infinite (e.g.~$K = \R$) and $d \not= 0$.
Taking $A$ as an alphabet, the configuration set $A^G$ inherits a natural product vector space structure.
The support of a configuration $x \in A^G$ is the subset $\supp(x) \coloneqq  \{g \in G: x(g) \neq 0 \} \subset G$.
Thus, $x \in A^G$ has finite support if and only if it is almost equal to the constant zero-configuration.
We denote by $A[G] \subset A^G$ the vector subspace consisting of all configurations with finite support.
\par
Recall that an \emph{involutive $K$-algebra}\index{involutive $K$-algebra} is a $K$-algebra equipped with an involution that is 
a $K$-algebra anti-automorphism.
\par 
The vector space $K[G]$ has a natural additional structure of an involutive $K$-algebra.
The multiplication on $K[G]$ is the \emph{convolution product}\index{convolution product} defined by
\[
(\alpha \beta)(g) \coloneqq
\sum_{\substack{g_1, g_2 \in G:\\ g_1g_2 = g}}  \alpha(g_1) \beta(g_2)
= \sum_{h \in G} \alpha(h) \beta(h^{-1} g)
\]
for all $\alpha, \beta \in K[G]$ and $g \in G$,
and  the involution is the map $\alpha \mapsto \alpha^*$ given by
 \[
 \alpha^*(g) \coloneqq  \alpha(g^{-1})
 \]
 for all $\alpha \in K[G]$ and $g \in G$.
This involutive $K$-algebra  is called the \emph{group algebra}\index{group!--- algebra} of the group $G$ with coefficients in $K$.
Note that the group $G$ embeds as a subgroup of the group of invertible elements of $K[G]$ via the map 
$g \mapsto \delta_g$, where $\delta_g \in K[G]$ is defined by $\delta_g(g) = 1$ and $\delta_g(h) = 0$ for all $h \in G$ with $h \not= g$,
and  that $G \subset K[G]$ is a base for the vector space $K[G]$.
\par
The $G$-shift action on $A^G$ is clearly $K$-linear, so that it yields a left $K[G]$-module structure on $A^G$. 
Observe that $A[G]$ is a submodule of $A^G$.
\par
A \emph{linear cellular automaton}\index{linear cellular automaton}\index{cellular automaton!linear ---} over the group $G$ and the alphabet $A$ is a cellular automaton 
$\tau \colon A^G \to A^G$ that is $K$-linear with respect to the vector space structure on $A^G$ 
(if $S$ is any memory set for $\tau$ and $\mu \colon A^S \to A$ is the associated local defining map, 
this is equivalent to requiring that $\mu$ is $K$-linear).
\par
Let us denote by $\LCA(G;A)$ the vector space consisting of all linear cellular automata $\tau \colon A^G \to A^G$.
\par
Let $\tau \in \LCA(G;A)$. 
Note that $A[G]$ is stable under $\tau$. Indeed, if $S \subset G$ is a memory set for $\tau$, then
$\supp(\tau(x)) \subset \supp(x)S^{-1}$ for all $x \in A^G$) (see~\cite[Proposition 8.2.3]{book}).
Moreover,  $\tau$ is pre-injective if and only if $\tau\vert_{A[G]} \colon A[G] \to A[G]$ is injective
(cf.~\cite[Proposition 8.2.5]{book}).
Observe also that $\tau$ is a $K[G]$-module endomorphism of $A^G$ and hence of $A[G]$.
\par
The vector space $\LCA(G;A)$ has a natural structure of a $K$-algebra with the composition of maps
as the multiplicative operation.
Furthermore, the restriction map $\tau \mapsto \tau\vert_{A[G]}$ yields a
$K$-algebra isomorphism from $\LCA(G;A)$ onto $\End_{K[G]}(A[G])$, the endomorphism $K$-algebra of the $K[G]$-module
$A[G]$ (cf.~\cite[Theorem 8.7.6]{book}).
It turns out  that $A[G]$ is a free $K[G]$-module with rank $d$.
Actually, if $(e_i)_{1 \leq i \leq d}$ is a base for the vector space $A$, then the family $(x_i)_{1 \leq i \leq d}$, where $x_i \in A[G]$ is the configuration defined by
$x_i(1_G) = e_i$ and $x_i(g) = 0$ for $g \not= 1_G$, is a free base for the $K[G]$-module $A[G]$
(cf.~\cite[Proposition 8.7.3]{book}).
One deduces that $\End_{K[G]}(A[G])$ is isomorphic, as a $K$-algebra,
to the $K$-algebra $\M_d(K[G])$ of $d \times d$ matrices with coefficients in the group algebra $K[G]$.
It follows that $\LCA(G;A)$ and $\M_d(K[G])$ are isomorphic as $K$-algebras 
(cf.~\cite[Corollary~8.7.8]{book}).
For instance, the map $\Phi \colon \M_d(K[G]) \to \LCA(G;A)$, sending each matrix
$\alpha = (\alpha_{i j})_{1 \leq i,j \leq d} \in \M_d(K[G])$ to the unique linear cellular automaton
$\tau \in \LCA(G;A)$ such that
\[
\tau(x_i) = \sum_{1 \leq j \leq d}  \alpha_{j i}^*  x_j
\]
for all $1 \leq i \leq d$, is a $K$-algebra isomorphism. 
\par
The \emph{adjoint}\index{adjoint matrix} of a matrix $\alpha \in \M_d(K[G])$ is the matrix
$\alpha^* \in \M_d(K[G])$ given by $(\alpha^*)_{ij} \coloneqq \alpha_{ji}^* \in K[G]$  for all 
$1 \leq i,j \leq d$.
The involution $\alpha \mapsto \alpha^*$ makes $\M_d(K[G])$ into an involutive $K$-algebra.
Let us  transport this involution to $\LCA(G;A)$ via $\Phi$.
Thus, $\LCA(G;A)$ becomes an involutive $K$-algebra with involution $\tau \mapsto \tau^*$ satisfying
\[
\tau^*(x_i) = \sum_{1 \leq j \leq d}  \alpha_{i j}  x_j
\]
for all $1 \leq i \leq d$ and $\tau = \Phi(\alpha) \in \LCA(G;A)$.
Note that this involution on $\LCA(G;A)$ depends on the choice of a base for $A$. 
\par
Consider now the non-degenerate $K$-bilinear symmetric map $A \times A \to K$ defined by
\[
a \cdot b = \sum_{1 \leq i \leq d} a_ib_i \quad \text{ for all  } a = \sum_{1 \leq i \leq d} a_i e_i \text{ and }   b = \sum_{1 \leq i \leq d} b_i e_i, \text{
with } a_i,b_i \in K \text{ for } 1 \leq i \leq d.
\]
Then the $K$-bilinear map $A[G] \times A^G \to K$, defined by
\[
\langle x,y \rangle \coloneqq \sum_{g \in G} x(g)\cdot y(g)
\]
for all $x \in A[G]$ and $y \in A^G$, is non-degenerate in both arguments.
Given a linear cellular automaton $\tau \colon A^G \to A^G$,
Bartholdi \cite{bartholdi:2017} (see also \cite{tointon-harmonic})\index{Bartholdi, Laurent} 
observed that
\[
\langle \tau(x), y \rangle = \langle x, \tau^*(y) \rangle
\]
for all $x \in A[G]$ and $y \in A^G$, and used this to show that $\tau$ is pre-injective (resp.\ surjective) if and only if
$\tau^*$ is surjective (resp.\ pre-injective).
\par
In~\cite{csc-linear-goe} (see also \cite[Theorem~8.9.6]{book})\index{Garden of Eden theorem!linear ---}\index{theorem!linear Garden of Eden ---} a linear version of the Garden of Eden theorem is proved, namely that
if $G$ is amenable and $\tau \in \LCA(G;A)$, then $\tau$ is surjective if and only if it is pre-injective.
\par
Let $G$ be a nonamenable group. In~\cite[Theorem 1.1]{bartholdi:2016} Bartholdi\index{Bartholdi, Laurent}  showed that there exists a finite field $K$ (in~\cite{bartholdi:2017} he actually observed that the field $K$ can be arbitrary), 
a finite dimensional vector space $A$ over $K$, 
and a pre-injective linear cellular automaton $\tau \in \LCA(G;A)$ which is not surjective. 
As a consequence (\cite[Corollary]{bartholdi:2017}), the cellular automaton $\tau^* \in \LCA(G;A)$ is surjective but not pre-injective.
These two facts, in combination with the linear version of the Garden of Eden theorem in~\cite{csc-linear-goe}, yield a
characterization of group amenability in terms of linear cellular automata.
\par
The linear version of the Garden of Eden theorem has been extended in~\cite{csc-goe-linear-sub} to linear cellular automata $\tau \colon X \to X$ with $X \subset A^G$ a strongly irreducible linear subshift of finite type, and in~\cite{csc-goe-modules} to the case when the alphabet $A$ is a semi-simple left-module of finite length
over a (possibly noncommutative) ring.

\subsection{Algebraic cellular automata}
Let $G$ be a group. In \cite{cc-algebraic-ca} we introduced the class of algebraic cellular automata over $G$. Given a field $K$, let $A$ be an \emph{affine algebraic set}\index{affine algebraic set} over $K$. This means that $A \subset K^n$ for some integer $n \geq 1$ is the set of common zeroes of a family of polynomials in $n$ variables with coefficients in $K$. Then a cellular automaton $\tau \colon A^G \to A^G$ is called an \emph{algebraic
cellular automaton}\index{algebraic cellular automaton}\index{cellular automaton!algebraic ---} provided it admits a memory set $S \subset G$ and a local defining map $\mu \colon A^S \to A$ that is
\emph{regular},\index{regular map} i.e., it is the restriction of some polynomial map $(K^n)^S \to K^n$. This definition was generalized in \cite{cscp-alg-ca}
as follows.
\begin{definition}
\label{def:ca-over-schemes}
Let ${\mathcal S}$ be a scheme and let $X, Y$ be schemes based over ${\mathcal S}$. 
Denote by $A \coloneqq X(Y)$ the set of $Y$-points of $X$,
that is, the set consisting of all ${\mathcal S}$-scheme morphism $Y \to X$.
Then an \emph{algebraic cellular automaton over the group $G$ and the
${\mathcal S}$-scheme $X$ with coefficients in the ${\mathcal S}$-scheme $Y$}, briefly, an \emph{algebraic cellular automaton over
the group $G$ and the schemes ${\mathcal S},X,Y$},\index{algebraic cellular automaton!--- over a scheme}\index{cellular automaton!algebraic --- over a scheme} is a cellular automaton $\tau \colon A^G \to A^G$ over the group $G$ and the alphabet $A$ that admits a memory set $S \subset G$ and a local defining map $\mu \colon A^S \to A$ which is induced by some ${\mathcal S}$-scheme morphism $f \colon X^S \to X$, where $X^S$ denotes the ${\mathcal S}$-fibered product of a family of copies of
$X$ indexed by $S$.
\end{definition}
Note that Definition \ref{def:ca-over-schemes} generalizes that of an algebraic cellular automaton given in~\cite{cc-algebraic-ca}.
Indeed, if $K$ is a field and $A \subset K^n$ an algebraic set, there is an ${\mathcal S}$-scheme $X$ associated with $A$ for
${\mathcal S} = \Spec(K)$, namely  $X = \Spec(K[u_1,\dots,u_n]/I)$, where $I = I(A)$ is the ideal of $K[u_1,\dots,u_n]$ consisting of all polynomials that identically vanish on $A$.
One then has $A = X({\mathcal S})$ and the regular maps between two regular sets $A_1 \subset K^{n_1}$ and $A_2 \subset K^{n_2}$
are precisely those induced by the ${\mathcal S}$-morphisms between their corresponding ${\mathcal S}$-schemes $X_1$ and $X_2$, equivalently, the $K$-algebra morphisms from
$K[z_1,\dots,z_{n_2}]/I(A_2)$ to $K[t_1,\dots,t_{n_1}]/I(A_1)$. Thus, $\tau \colon A^G \to A^G$ is an algebraic cellular automaton,
as defined in~\cite{cc-algebraic-ca}, if and only if $\tau$ is a cellular automaton in the sense of Definition~\ref{def:ca-over-schemes}
over the schemes ${\mathcal S},X,Y$ for ${\mathcal S} = Y = \Spec(K)$ and $X$ is the ${\mathcal S}$-scheme associated with $A$.
Recall that an \emph{algebraic variety}\index{algebraic variety} over a field $K$ is a scheme of finite type over $K$.
In \cite[Theorem 1.1]{cscp-GOE-myhill} we showed the following Myhill type result for algebraic cellular automata:
\begin{theorem}
\label{t:main-theorem-1-cscp}
Let $G$ be an amenable group and let $X$ be an irreducible complete algebraic variety over an algebraically closed field $K$.
Let $A \coloneqq X(K)$ denote the set of $K$-points of $X$.
Then every pre-injective algebraic cellular automaton $\tau \colon A^G \to A^G$
over $(G,X,K)$ is surjective.
\end{theorem}
Let us note that the converse implication, i.e., the analogue of the Moore implication, does not hold under the hypotheses of
Theorem~\ref{t:main-theorem-1-cscp}, even with the additional hypothesis that the variety $X$ is complete. 
For example, if $K$ is an algebraically closed field whose characteristic is not equal to $2$,
the projective line $\Proj_K^1$ is an irreducible complete $K$-algebraic variety
and the morphism $f \colon \Proj_K^1 \to \Proj_K^1$ given by $(x:y) \mapsto (x^2:y^2)$ is surjective but not injective.
Taking $A \coloneqq \Proj_K^1(K)$,  we deduce that, for any group $G$, the map $\tau \colon A^G \to A^G$ defined by
$(\tau(c))(g) \coloneqq  f(c(g))$ for all $c \in A^G$ and $g \in G$, is an algebraic cellular automaton over $(G,X,K)$
that is surjective but not pre-injective.
\par
In order to formulate a version of the Garden of Eden theorem for algebraic cellular automata, the following weak notion of pre-injectivity
was introduced in \cite[Definition 6.1]{cscp-GOE-myhill}:
\begin{definition}
\label{def:*}\index{algebraic cellular automaton!$(*)$-pre-injective ---}\index{cellular automaton!$(*)$-pre-injective algebraic ---}
Let $G$ be a group and let $X$ be an algebraic variety over an algebraically closed field $K$.
Let $A\coloneqq X(K)$ and let $\tau \colon A^G \to A^G$ be an algebraic cellular automaton over $(G,X,K)$.
\par
We say that $\tau$ is \emph{$(*)$-pre-injective} if there do not exist a finite subset $\Omega \subset G$
and a proper subset $H \subset A^\Omega$ that is closed for the Zariski topology such that
\[
\tau((A^\Omega)_p)=\tau(H_p) \quad \text{ for all } p \in A^{G\setminus \Omega}
\]
where $H_p \coloneqq \{x \in A^G: x\vert_\Omega \in H \mbox{ and } x\vert_{G \setminus \Omega} = p\}$ for $p \in A^{G\setminus \Omega}$ and any subset $H \subset A^\Omega$.
\end{definition}
It turns out that Theorem~\ref{t:main-theorem-1-cscp} remains valid if we replace the hypothesis that $\tau$ is pre-injective by the weaker hypothesis that $\tau$ is $(*)$-pre-injective. Moreover, this weak form of pre-injectivity also allows us to establish a version of the Moore implication for algebraic cellular automata. Altogether we obtained the following version of the Garden of Eden theorem
(cf.\ \cite[Theorem 1.4]{cscp-GOE-myhill}) for algebraic cellular automata:
\begin{theorem}
\label{t:main-theorem-2-cscp}\index{Garden of Eden theorem!--- for algebraic cellular automata}\index{theorem!Garden of Eden --- for algebraic cellular automata}
Let $G$ be an amenable group and let $X$ be an irreducible complete algebraic variety over an algebraically closed field $K$.
Let $A \coloneqq X(K)$ denote the set of $K$-points of $X$ and let $\tau \colon A^G \to A^G$ be an algebraic cellular automaton
over $(G,X,K)$.
Then the following conditions are equivalent:
\begin{enumerate}[\rm(a)]
\item
$\tau$ is surjective;
\item
$\tau$ is $(*)$-pre-injective.
\end{enumerate}
\end{theorem}
One of the main ingredients in the proof of the above results is the notion of \emph{algebraic mean dimension}.\index{algebraic mean dimension}
If $G$ is an amenable group equipped with a F\o lner net $\FF$ and $A$ is the set of $K$-points of an algebraic variety $X$ over an algebraically closed field $K$, given a subset $\Gamma \subset A^G$ then the algebraic mean dimension
$\mdim_\FF(\Gamma)$ of $\Gamma$ is defined as a limit of the average \emph{Krull dimension}\index{Krull dimension} 
of the projection of $\Gamma$ along the F\o lner net.
The definition of algebraic mean dimension is analogous to that of topological entropy (cf.\ \eqref{e:entropy}).

\subsection{Gromov's Garden of Eden theorem}\index{Gromov's Garden of Eden theorem}\index{Garden of Eden theorem!Gromov's ---}\index{theorem!Gromov's Garden of Eden ---}
In \cite[Subsection~8.F']{gromov-esav}, Gromov proved a Garden of Eden type theorem generalizing Theorem~\ref{t:GOE-amen} under several
aspects. First of all, the alphabet set $A$ is only assumed to be countable, not necessarily finite.  
In addition, the universe
is (the vertex set $V$ of) a connected simplicial graph $\GG = (V,E)$ of bounded degree with a
natural homogeneity condition (to admit a \emph{dense pseudogroup of partial isometries}).
The classical case corresponds to $\GG = \CC(G,S)$ being the Cayley graph of a finitely generated group $G$ with respect to a finite and
symmetric generating subset $S \subset G \setminus \{1_G\}$. The dense pseudogroup of partial isometries is, in this particular case, given by partial left-multiplication by group elements.
In this more general setting, the category corresponding to that of cellular automata consists now of the following:
\begin{itemize}
\item \emph{stable spaces},\index{stable space} i.e.\ (stable) projective limits of locally-finite projective systems $(X_\Omega)$ of $A$-valued maps on  (subsets of) $V$ with a
suitable finiteness and irreducibility condition (\emph{bounded propagation})\index{stable space!--- of bounded propagation}\index{bounded propagation!--- of a stable space} and admitting a \emph{dense holonomy}\index{stable space!--- of dense holonomy}\index{dense holonomy!--- of a stable space} (corresponding to shift-invariance in the classical case) as objects, and
\item \emph{maps of bounded propagation}\index{map of bounded propagation}\index{bounded propagation!map of ---}
(this condition corresponds to continuity) admitting a \emph{dense holonomy}\index{map of bounded propagation!--- admitting a dense holonomy}\index{dense holonomy!map of bounded propagation admitting a ---} (this corresponds to $G$-equivariance), as morphisms.
\end{itemize}
The notions of a F\o lner sequence and of amenability for simplicial graphs, together with the corresponding notion of entropy (for the above-mentioned spaces of $A$-valued maps), carry verbatim from the group theoretical framework.
All this said, Gromov's theorem states the following.\index{Gromov's Garden of Eden theorem}\index{Garden of Eden theorem!Gromov's ---}\index{theorem!Gromov's Garden of Eden ---}

Let $\GG = (V.E)$ be an amenable simplicial connected graph of bounded degree admitting a dense pseudogroup of partial isometries and let $A$ be a finite or
countably infinite alphabet set.  Suppose that $X,Y \subset A^V$ are stable spaces of bounded propagation with the same entropy.
Let $\tau \colon X \to Y$ be a a map of bounded propagation admitting a dense holonomy.
Then $\tau$ is surjective if and only if it is pre-injective.

In \cite[Lemma~3.11]{CFS-goe} it is shown that a stable space of bounded propagation is strongly irreducible (cf.\ Section \ref{sec:goe-subshifts}) and of finite type (cf.\ Section~\ref{sec:configur-shifts}).
However, as shown in \cite[Counterexample~3.13]{CFS-goe}, the converse fails to hold: strong irreducibility and finite type conditions do not
imply, in general, bounded propagation.
As a consequence, the following theorem (cf.\ \cite[Theorem~B]{CFS-goe}) improves on 
Gromov's theorem.

Let $\GG = (V.E)$ be an amenable simplicial connected graph of bounded degree admitting a dense pseudogroup of partial isometries and let $A$ be a finite or countably infinite alphabet set. Suppose that $X,Y \subset A^V$ are strongly irreducible stable spaces of finite type with the same entropy.
Let $\tau \colon X \to Y$ be a a map of bounded propagation admitting a dense holonomy.
Then $\tau$ is surjective if and only if it is pre-injective.

Note that this last result also covers Theorem \ref{t:fiorenzi-strongly}.

\subsection{Cellular automata over homogeneous sets}
Cellular automata where the universe is a set endowed with a transitive group action have been investigated by Moriceau~\cite{moriceau}.
Versions  of the Curtis-Hedlund-Lyndon theorem and of the Garden of Eden theorem in this more general setting have been obtained by
Wacker~\cite{wacker-curtis},~\cite{wacker-goe}. \\

\noindent
{\bf Acknowledgements.} We thank Laurent Bartholdi for valuable comments and remarks.

\bibliographystyle{siam}
\def\cprime{$'$}

\printindex
\end{document}